\documentclass[reqno]{amsart}

\usepackage{microtype}
\usepackage{hyperref}
         
\usepackage{amsmath}             
\usepackage{amsfonts}             
\usepackage{amsthm}               
\usepackage{amsbsy}
\usepackage{amssymb}
\usepackage{amsthm}
\usepackage{amsrefs}
\usepackage{multirow}

\usepackage{enumerate}

\usepackage{tikz}
\usetikzlibrary{shapes,snakes,calc,arrows}
\usetikzlibrary{decorations.pathreplacing,shapes.geometric}
\usetikzlibrary{calc,positioning}
\tikzset{Box/.style={very thick, rounded corners}}
\tikzset{marked/.style={star, star point height = .75mm, star points =5, fill=black,minimum size=2mm, inner sep=0mm} }
\tikzset{verythickline/.style = {line width=7pt}}
\tikzset{thickline/.style = {line width=5pt}}
\tikzset{medthick/.style = {line width=3pt}}
\tikzset{med/.style = {line width=2pt}}
\tikzset{count/.style = {fill=white,circle,draw,thin, inner sep=2pt}}
\tikzset{rcount/.style = {fill=white,rectangle,draw,thin,inner sep=2pt, rounded corners}}
\tikzset{cpr/.style = {draw,fill=white,rectangle,thin, rounded corners}}

\newtheorem{thm}{Theorem}[section]
\newtheorem{lem}[thm]{Lemma}

\newtheorem{cor}[thm]{Corollary}

\theoremstyle{definition}
\newtheorem{defn}[thm]{Definition}

\newtheorem{rem}[thm]{Remark}

\newcommand{\bC}{{\mathbb{C}}}
\newcommand{\bN}{{\mathbb{N}}}
\newcommand{\bR}{{\mathbb{R}}}

\newcommand{\bZ}{{\mathbb{Z}}}

\newcommand{\A}{{\mathcal{A}}}
\newcommand{\B}{{\mathcal{B}}}

\renewcommand{\H}{{\mathcal{H}}}
\newcommand{\K}{{\mathcal{K}}}

\renewcommand{\P}{{\mathcal{P}}}

\renewcommand{\phi}{\varphi}

\newcommand{\qqand}{\qquad\text{and}\qquad}

\newcommand{\vs}{{\mathrm{vs}}}
\newcommand{\tr}{{\mathrm{tr}}}
\newcommand{\op}{{\mathrm{op}}}

\newcommand{\NC}{{\mathrm{NC}}}
\newcommand{\BNC}{{\mathrm{BNC}}}

\author{Paul Skoufranis}
\address{Department of Mathematics and Statistics, York University, 4700 Keele Street, Toronto, Ontario, M3J 1P3, Canada}
\email{pskoufra@yorku.ca}
\date{\today}
\thanks{This research was completed with the support of NSERC (Canada) grant RGPIN-2024-06269.}

\begin{document}

\nocite{*}

\title[CLT for Tensor Products of Free Variables via Bi-Free]{Central Limit Theorem for Tensor Products of Free Variables via Bi-Free Independence}

\begin{abstract}
In this paper, a connection between bi-free probability and the asymptotics of random quantum channels and tensor products of random matrices is established.  Using bi-free matrix models, it is demonstrated that the spectral distribution of certain self-adjoint quantum channels and tensor products of random matrices  tend to a distribution that can be obtained by an averaged sum of products of bi-freely independent pairs.  Subsequently, using bi-free techniques, a Central Limit Theorem for such operator is established.
\end{abstract}

\maketitle

\section{Introduction}

One essential concept in Quantum Information Theory is that of quantum channels which are mathematical/physical models for how quantum information can be transmitted.  Given a quantum channel $\Phi : M_n(\bC) \to M_n(\bC)$, the (non-unique) Kraus decomposition of $\Phi$ implies there exists $\{K_j\}^d_{j=1} \subseteq M_n(\bC)$ such that
\[
\Phi(X) = \sum^d_{j=1} K_j X K_j^*
\]
for all $X \in M_n(\bC)$.  Viewing $M_n(\bC) \cong \bC^n \otimes \bC^n$, it is natural to examine $\Phi$ through its Kraus operator
\[
K_\Phi = \sum^d_{j=1} K_j \otimes \overline{K_j}
\]
where $\overline{K_j}$ denotes the entry-wise conjugation of $K_j$.  Furthermore, $\{K_j\}^d_{j=1}$ can be assumed to be self-adjoint when $\Phi$ is self-adjoint in the sense that $\tr(X\Phi(Y)) = \tr(\Phi(X)Y)$ for all $X,Y \in M_n(\bC)$.

In \cite{LSY2023} the asymptotic behaviour of random self-adjoint quantum channels was examined. This was accomplished by fixing a $d$, taking the above $\{K_j\}^d_{j=1}$ to be random self-adjoint $n \times n$ matrices, and examining the behaviour of the Kraus operator as $n$ tends to infinity.  In particular, the following was demonstrated.

\begin{thm}[\cite{LSY2023}*{Theorem 2.2}]\label{thm:quantum-channel}
Let $W_1, \ldots, W_d$ be centred, self-adjoint random $n \times n$ matrices such that
\begin{itemize}
\item $W_j$ converges weakly in probability and expectation to $\mu_j$ as $n$ tends to infinity for all $j$, 
\item $\{W_j\}^d_{j=1}$ are in probability and expectation asymptotically free as $n$ tends to infinity, and 
\item $\mathbb{E}\left(W_j \otimes \overline{W_j}\right)$ converges weakly to 0 as $n$ tends to infinity for all $j$.
\end{itemize}
Let $\Phi_d$ be the quantum channel with Kraus operator
\[
\sum^d_{j=1} \frac{1}{\sqrt{d}} W_j \otimes \frac{1}{\sqrt{d}}\overline{W_j},
\]
let $a_1, \ldots, a_d$ be freely independent random variables with respect to $\varphi$ with distributions $\mu_1, \ldots, \mu_d$ respectively, and let
\[
\Delta_{d,n} = \frac{1}{\sqrt{d}} \sum^d_{j=1}\left(W_j \otimes \overline{W_j} -\mathbb{E}\left(W_j \otimes \overline{W_j}\right)\right).
\]
Then in probability and expectation the distribution of $\Delta_{d,n}$ (and thus the spectral distribution of $\Phi_d- \mathbb{E}(\Phi_d)$) tends to the distribution of $\frac{1}{\sqrt{d}} \sum^d_{j=1} a_j \otimes a_j$ with respect to $\varphi \otimes \varphi$ as $n$ tends to infinity.  

Moreover, if $d = d(n)$ diverges as $n$ tends to infinity and $\{W_j\}^d_{j=1}$ are normalized, independent, and identically distributed, then in probability and expectation the distribution of $\Delta_{d,n}$ (and thus the spectral distribution of $\Phi_d- \mathbb{E}(\Phi_d)$)  tends to the centred semicircular distribution of variance 1 as $n$ tends to infinity.  
\end{thm}

To obtain the semicircular distribution, \cite{LSY2023} demonstrated that the distribution of $\frac{1}{\sqrt{d}} \sum^d_{j=1} a_j \otimes a_j$ with respect to $\varphi \otimes \varphi$ tends to the semicircular distribution as $d$ tends to infinity.  Furthermore, the same authors in \cite{LSY2024} developed a Central Limit Theorem for tensor products of freely independent operators in the case where the $a_j$ need not be centred (see Theorem \ref{thm:main}).  The above results led \cite{LSY2024} to pose the question of whether there was a general notion of independence corresponding to such objects where properties and limits could be derived.

In \cite{V2014} Voiculescu introduced the notion of bi-free independence.  Since then the theory has been quite extensively developed (e.g. \cites{BBGS2018, CNS2015-1, CNS2015-2, CS2020, CS2021, S2016, S2017}).  Unlikely other notions of independence, bi-free independence is a notion of independence for pairs of algebras.  In particular, if $(\A_1, \B_1)$ and $(\A_2, \B_2)$ are bi-freely independent, then $\A_1$ and $\A_2$ are freely independent, $\B_1$ and $\B_2$ are freely independent, $\A_1$ and $\B_2$ are classically independent, and $\A_2$ and $\B_1$ are classically independent.  Although the converse need not hold in general (see \cite{S2016}), bi-free independence is a notion of independence where free and classical independence intermingle.

The goal of this paper is to demonstrate that bi-free independence is indeed an appropriate notion of independence where properties of sums of tensor products of freely independent operators can be studied, and properties and limits can be derived.   In particular, we will demonstrate that the results mentioned above can be verified using bi-free techniques in a more direct fashion.  Consequently, this paper establishes a direct connection between bi-free probability and the asymptotics of random quantum channels and tensor products of random matrices.  

This paper also directly advances the techniques of bi-free probability.  In the current literature involving bi-free independence, the left and right operators remain isolated from one another.  For example, the bi-free cumulants (see Section \ref{sec:background}) distinguish which entries are left and which entries are right (mixing is not allowed), all bi-free transformations are two variable transformations with one variable for left operators and one variable for right operators, and the existing notion of bi-free linearization (\cite{BBGS2018}) can only handle the joint distribution of a polynomial in left variables and a polynomial in right variable (that is, it cannot handle a single polynomial in left and right variables).  The only mixing of left and right operators that has occurred was with bi-free conjugate variables in \cite{CS2020} as the bi-free conjugate variables were in the $L_2$-space generated by the left and right operators, but this did not yield any spectral information about polynomials in left and right operators.  This work shows that sums of products of left and right operators can be studied using bi-free techniques thereby yielding hope the theory can be used to examine arbitrary polynomials in left and right operators.  

Excluding this introduction, this paper has four sections structured as follows.  Section \ref{sec:background} reviews the few basic bi-free techniques required for this paper.  It is explained how tensor products of freely independent operators can be examined via bi-free independence, and the basics of bi-free cumulants (which behave in a similar fashion to free cumulants) are summarized.  Section \ref{sec:matrix-models} provides a condensed view of bi-free matrix models and uses bi-free results to obtain a different version of the fixed $d$ portion of Theorem \ref{thm:quantum-channel}.  In particular, Theorem \ref{thm:bi-matrix} drops the centred assumption from Theorem \ref{thm:quantum-channel} at the cost of slightly modifying the $\Delta_{d,n}$ operator.

Sections \ref{sec:centred} and \ref{sec:general} examine the Central Limit Theorem for tensor products of freely independent operators from \cite{LSY2024} through the lens of bi-free independence.  Theorem \ref{thm:centred-result} provides a short proof of this Central Limit Theorem in the centred case using bi-free techniques and a connection with meandric systems.  Theorem \ref{thm:main} independently establishes the Central Limit Theorem from \cite{LSY2024} in the general setting using bi-free techniques.

\section{Background on Bi-Free Independence}
\label{sec:background}

To study bi-free independence, some structures used in free independence are required.  For a natural number $n$, let $[n] = \{1,2, \ldots, n\}$.  A partition $\pi = \{V_1, \ldots, V_m\}$ of $[n]$ is a collection of disjoint sets $V_1, \ldots, V_m$, called blocks of $\pi$, such that
\[
[n] = V_1 \cup V_2 \cup \cdots \cup V_m.
\]
Let $\P(n)$ denote the set of all partitions on $[n]$.  Given $\pi, \sigma \in \P(n)$, it is said that $\sigma$ is a \emph{refinement} of $\pi$, denoted $\sigma \leq \pi$, if every block of $\sigma$ is contained in a single block of $\pi$.

A partition $\pi$ on $[n]$ is said to be \emph{non-crossing} if whenever there are blocks $V, W \in \pi$ with $v_1, v_2 \in V$ and $w_1, w_2 \in W$ such that $v_1 < w_1 < v_2 < w_2$, then $V = W$.   The set of all non-crossing partitions on $[n]$ will be denoted $\NC(n)$. A non-crossing partition $\pi$ is said to be a \emph{pair non-crossing partition} if every block of $\pi$ contains precisely two elements.  The set of all pair non-crossing partitions on $[n]$ is denoted $\NC_2(n)$.  Note $n$ must be even for $\NC_2(n)$ to be non-empty.

For the purposes of this paper, a \emph{non-commutative probability space} is a pair $(\A, \varphi)$ where $\A$ is a unital C$^*$-algebra and $\varphi : \A \to \bC$ is a unital positive linear map.  We do not require $\varphi$ to be tracial as it can be observed the state in the following definition of bi-free independence need not be tracial.

\begin{defn}[\cite{V2014}]\label{defn:bi-free}
Let $(\A, \varphi)$ be a non-commutative probability space.  Pairs of unital C$^*$-subalgebras $\{(\A_{\ell,k}, \A_{r, k})\}_{k \in K}$ of $\A$ are said to be \emph{bi-freely independent with respect to $\varphi$} if there exist Hilbert spaces $\H_k$, unit vectors $\xi_k \in \H_k$, and unital $*$-homomorphisms $\alpha_k : \A_{\ell, k} \to \mathcal{B}(\H_k)$ and $\beta_k : \A_{r,k} \to \mathcal{B}(\H_k)$ such that if $\H = \ast_{k \in K} (\H_k, \xi_k)$ is the reduced free product space, $\lambda_k, \rho_k : \B(\H_k) \to \B(\H)$ are the left and right regular representations respectively, and $\varphi_0 : \B(\H) \to \bC$ is the vacuum vector state, the joint distribution of $\{(\A_{\ell,k}, \A_{r, k})\}_{k \in K}$ with respect to $\varphi$ is equal to the joint distribution of $\{(\lambda_k(\alpha_k(\A_{\ell,k})), \rho_k(\beta_k(\A_{r, k})))\}_{k \in K}$ with respect to $\varphi_0$.
\end{defn}

Although the above definition is very general, this paper will only require a certain class of bi-free independent pairs.  It is elementary from the definitions of free and bi-free independence that unital $*$-algebras $\A_1, \ldots, \A_n \subseteq \A$ are free with respect to $\varphi$ if and only if $\{(\A_k, \bC)\}^n_{k=1}$ are bi-free with respect to $\varphi$ if and only if $\{(\bC, \A_k)\}^n_{k=1}$ are bi-free with respect to $\varphi$.  Thus the notion of free independence embeds into bi-free independence.  Moreover, it is not difficult to see classical independence (also known as tensor independence) embeds into the bi-free setting.

\begin{rem}\label{rem:tensor-independence}
Let $(\A, \varphi)$ be a non-commutative probability space.  To simplify the need for representations, assume $\A \subseteq \B(\H)$ for some Hilbert space $\H$ and $\xi \in \H$ is a unit vector so that $\varphi$ is the vector state corresponding to $\xi$.  Let $\H_k = \H$, $\xi_k = \xi$ for $k=1,2$, and $\K = (\H_1, \xi_1) \ast (\H_2, \xi_2)$ with vacuum vector $\xi_0$.  By the definition of the reduced free product and the left and right regular representations, it is easy to verify that $\lambda_1 : \B(\H_1) \to \B(\K)$ and $\rho_2 : \B(\H_2) \to \B(\K)$ commute with each other, $\lambda_1(\B(\H_1)) \rho_2(\B(\H_2)) \xi_0 = \H_1 \otimes \H_2 \subseteq \K$, and
\[
\langle \lambda_1(T)\rho_2(S) \xi_0, \xi_0\rangle = \varphi(T) \varphi(S)
\]
for all $T, S \in \B(\H)$.  Consequently, $\lambda_1 \otimes \rho_2 : \A \otimes \A \to \B(\K)$ is a representation of $\A \otimes \A$ such that the vector state defined by the vacuum vector $\xi_0$ is $\varphi \otimes \varphi$.  Consequently  $(\A \otimes \bC, \bC\otimes \bC)$ and $(\bC\otimes \bC, \bC \otimes \A)$ are bi-free pairs with respect to $\varphi \otimes \varphi$.
\end{rem}

By combining the above and the associativity of bi-free independence, we obtain the following which is the only collection bi-free operators required in this paper.

\begin{lem}\label{lem:getting-the-bi-free-we-need}
Let $(\A, \varphi)$ be a non-commutative probability space.  Let $a_1, \ldots, a_n \in \A$ be free with respect to $\varphi$ and $b_1, \ldots, b_n \in \A$ be free with respect to $\varphi$.  Then
\[
\{(a_k \otimes 1, 1\otimes 1)\}^n_{k=1} \cup \{(1\otimes 1, 1\otimes b_k)\}^n_{k=1}
\]
are bi-free in $(\A \otimes \A, \varphi \otimes \varphi)$.
\end{lem}

It is not difficult to see bi-free independence is associative by examining the representations in Definition \ref{defn:bi-free}.  Alternatively, bi-free cumulants can also be used to demonstrate associativity.  Luckily the combinatorial approach to bi-free independence is very similar to that of free probability (see \cite{NS2006}) where a small modification to the set of partitions is required.  In particular, \cite{S2016} uses bi-free cumulants to establish  Remark \ref{rem:tensor-independence}.

To summarize the combinatorics of bi-free independence, a map $\chi: [n] \to \{\ell, r\}$ will be used to designate whether each operator from a set of $n$ operators should be viewed as a left operator ($\ell$) or as a right operator ($r$). Given $\chi$, write 
\[
\chi^{-1}(\{\ell\}) = \{i_1<\cdots<i_p\} \qqand \chi^{-1}(\{r\}) = \{i_{p+1} > \cdots > i_n\},
\]
and define the permutation $s_\chi$ on $[n]$ by $s_\chi(k) = i_k$ for all $k$.  In addition, we define the total ordering $\prec_\chi$ on $[n]$ by $a\prec_\chi b$ if and only if $s_\chi^{-1} (a)< s_\chi^{-1}(b)$.  Notice $\prec_\chi$ corresponds to, instead of reading $[n]$ in the traditional order, reading $\chi^{-1}(\{\ell\})$ in increasing order followed by reading $\chi^{-1}(\{r\})$ in decreasing order.  

\begin{defn}
A partition $\pi \in \P(n)$ is said to be \emph{bi-non-crossing with respect to $\chi$} if the partition  $s_\chi^{-1}\cdot \pi$ (the partition formed by applying $s_\chi^{-1}$ to the blocks of $\pi$) is non-crossing.
Equivalently $\pi$ is bi-non-crossing if whenever there are blocks $V, W \in \pi$ with $v_1, v_2 \in V$ and $w_1, w_2 \in W$ such that $v_1 \prec_\chi w_1 \prec_\chi v_2 \prec_\chi w_2$, then $V = W$.  The set of bi-non-crossing partitions with respect to $\chi$ is denoted by $\BNC(\chi)$.
\end{defn}

Alternatively, one can determine whether a partition $\pi \in \P(n)$ is bi-non-crossing with respect to $\chi$ via the following diagrammatic process:  place nodes along two dashed vertical lines, labelled $1$ to $n$ from top to bottom, such that the nodes on the left dashed line correspond to those values for which $\chi(k) = \ell$ and nodes on the right dashed line correspond to those values for which $\chi(k) = r$.  Then, between the two vertical dashed lines, use solid lines to connect the nodes which are in the same block of $\pi$.  Then $\pi$ is bi-non-crossing with respect to $\chi$ exactly when the solid lines can be drawn to not intersect (i.e. cross).  For example, it is easy to see that $\pi = \{\{1,4\}, \{2,5\}, \{3,6\}\}$ has crossings but $\pi \in \BNC(\chi)$ when $\chi^{-1}(\{\ell\}) = \{1,4,5\}$ and $\chi^{-1}(\{r\}) = \{2,3,6\}$:

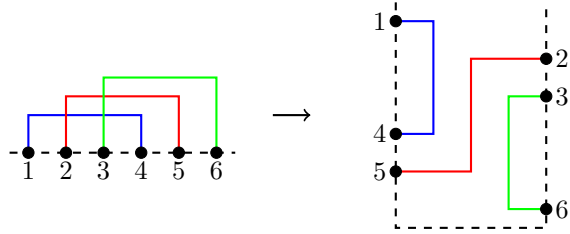
\begin{align*}
\begin{tikzpicture}[baseline]
	\draw[thick, dashed] (-0.25,1) --(2.75,1);
	\draw[thick, blue] (0,1) -- (0, 1.5) -- (1.5, 1.5) -- (1.5, 1);
	\draw[thick, red] (0.5,1) -- (0.5, 1.75) -- (2, 1.75) -- (2, 1);
	\draw[thick, green] (1,1) -- (1, 2) -- (2.5, 2) -- (2.5, 1);
	\draw[fill=black] (0, 1) circle (0.075);
	\draw[fill=black] (0.5, 1) circle (0.075);
	\draw[fill=black] (1, 1) circle (0.075);
	\draw[fill=black] (1.5, 1) circle (0.075);
	\draw[fill=black] (2, 1) circle (0.075);
	\draw[fill=black] (2.5, 1) circle (0.075);
	\node[below] at (0, 1) {$1$};
	\node[below] at (0.5, 1) {$2$};
	\node[below] at (1, 1) {$3$};
	\node[below] at (1.5, 1) {$4$};
	\node[below] at (2, 1) {$5$};
	\node[below] at (2.5, 1) {$6$};
	\draw[thick, ->] (3.25, 1.5) -- (3.75, 1.5);
	\end{tikzpicture}
\qquad
\begin{tikzpicture}[baseline]
	\draw[thick, dashed] (-1,3) -- (-1,0) -- (1, 0) -- (1,3);
	\draw[thick, blue] (-1, 2.75) -- (-0.5, 2.75) -- (-0.5, 1.25) -- (-1, 1.25);
	\draw[thick, red] (1, 2.25)  -- (0, 2.25)  -- (0, .75) -- (-1, 0.75);
	\draw[thick, green] (1, 0.25) -- (0.5, 0.25) -- (.5, 1.75) -- (1, 1.75);
	\draw[fill=black] (-1, 2.75) circle (0.075);
	\node[left] at (-1, 2.75) {$1$};
	\draw[fill=black] (1, 2.25) circle (0.075);
	\node[right] at (1, 2.25) {$2$};
	\draw[fill=black] (1, 1.75) circle (0.075);
	\node[right] at (1, 1.75) {$3$};
	\draw[fill=black] (-1, 1.25) circle (0.075);
	\node[left] at (-1, 1.25) {$4$};
	\draw[fill=black] (-1, .75) circle (0.075);
	\node[left] at (-1, .75) {$5$};
	\draw[fill=black] (1, 0.25) circle (0.075);
	\node[right] at (1, 0.25) {$6$};
\end{tikzpicture}
\end{align*}

To obtain the bi-free cumulants, we require a M\"{o}bius inversion of the moments.  The \emph{bi-non-crossing M\"{o}bius function} is the function
\[
\mu_{\BNC} : \bigcup_{n\geq1}\bigcup_{\chi : [n] \to\{\ell, r\}} \BNC(\chi)\times \BNC(\chi) \to \bZ
\]
defined such that $\mu_{\BNC}(\pi, \sigma) = 0$ unless $\pi$ is a refinement of $\sigma$, and otherwise defined recursively via the formulae
\[
\sum_{\substack{\tau \in \BNC(\chi) \\\pi \leq \tau \leq \sigma}} \mu_{\BNC}(\tau, \sigma) = \sum_{\substack{\tau \in \BNC(\chi) \\ \pi \leq \tau \leq \sigma}} \mu_{\BNC}(\pi, \tau) = \begin{cases}
1 & \mbox{if } \pi = \sigma  \\
0 & \mbox{otherwise}
\end{cases}.
\]
Due to the similarity of the lattice structures, the bi-non-crossing M\"{o}bius function is related to the non-crossing M\"{o}bius function $\mu_{\NC}$ via the formula 
\[
\mu_{\BNC}(\pi, \sigma) = \mu_{\NC}(s^{-1}_\chi \cdot \pi, s^{-1}_\chi \cdot \sigma).
\]

To define the bi-free cumulants, fix a non-commutative probability space $(\A, \varphi)$ and a map $\chi : [n] \to \{\ell, r\}$.  Given $\pi \in \BNC(\chi)$ and $Z_1,\ldots, Z_n \in \A$, let
\[
\varphi_\pi(Z_1, \ldots, Z_n) = \prod_{V \in \pi} \varphi\left( \prod_{q \in V} Z_q\right)
\]
where the product (like all products in this paper) is performed in increasing order. 

\begin{defn}
The \emph{bi-free cumulant of $Z_1,\ldots, Z_n$ with respect to $\chi$} is 
\begin{align}\label{eq:cumulant}
\kappa_\chi(Z_1, \ldots, Z_n) = \sum_{\pi \in \BNC(\chi)} \varphi_\pi(Z_1, \ldots, Z_n) \mu_{\BNC}(\pi, 1_{n}),
\end{align}
where $1_n = \{[n]\}$ denotes the full partition.  
\end{defn}

If for $\pi \in \BNC(\chi)$ we define
\[
\kappa_\pi(Z_1, \ldots, Z_n) = \prod_{V \in \pi} \kappa_{\chi|_V}\left((Z_1, \ldots, Z_{n})|_{V}\right)
\]
where $(Z_1, \ldots, Z_{n})|_{V}$ denotes restricting the $n$-tuple to elements of $V$, we obtain by M\"{o}bius inversion that
\begin{align}\label{eq:moment}
\varphi(Z_1 \cdots Z_n) = \varphi_{1_n}(Z_1, \ldots, Z_n) = \sum_{\pi \in \BNC(\chi)} \kappa_\pi(Z_1, \ldots, Z_n). 
\end{align}
Equations (\ref{eq:cumulant}) and (\ref{eq:moment}) are known as the bi-free moment-cumulant formulae.  

In the same way free cumulants characterize free independence, bi-free cumulants characterize bi-free independence.
\begin{thm}[\cites{CNS2015-1, CNS2015-2}]\label{thm:mixed-cumulants-vanish}
Let $\{(\A_{\ell, k}, \A_{r, k})\}_{k \in K}$ be pairs of unital C$^*$-subalgebras of a non-commutative probability space $(\A, \varphi)$.  Then $\{(\A_{\ell, k}, \A_{r, k})\}_{k \in K}$ are bi-freely independent with respect to $\varphi$ if and only if mixed bi-free cumulants vanish; that is, for all $\chi : [n]\to\{\ell,r\}$, $\theta : [n] \to K$ non-constant, and $Z_m \in A_{\chi(m), \theta(m)}$,
\[
\kappa_\chi(Z_1, \ldots, Z_n) = 0.
\]
\end{thm}

For the arguments of this paper, we will be able to greatly simplify and reduce the combinatorics of bi-non-crossing partitions.   In particular, \cite{S2016} used the following sub-collection of bi-non-crossing partitions to study classical independence.

\begin{defn}
Given a map $\chi : [n] \to \{\ell, r\}$, a bi-non-crossing partition $\pi \in \BNC(\chi)$ is said to be \emph{vertically split} if whenever $V$ is a block of $\pi$, either $V \subseteq \chi^{-1}(\{\ell\})$ or $V \subseteq \chi^{-1}(\{r\})$. The set of vertically split bi-non-crossing partitions is denoted by $\BNC_{\vs}(\chi)$.
\end{defn}

Note $\pi \in \BNC(\chi)$ is vertically split exactly when a vertical line can be drawn to split the left and right sides of the bi-non-crossing diagram of $\pi$.  Consequently, $\BNC_{\vs}(\chi)$ is in bijective correspondence with pairs of non-crossing partitions; one on the left and one on the right.

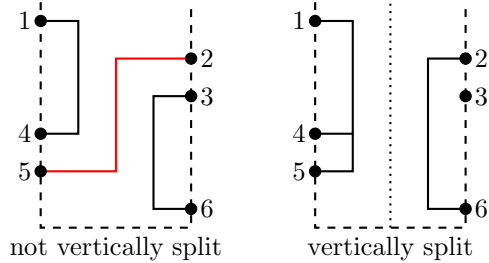
\begin{align*}
\begin{tikzpicture}[baseline]
	\draw[thick, dashed] (-1,3) -- (-1,0) -- (1, 0) -- (1,3);
	\draw[thick, black] (-1, 2.75) -- (-0.5, 2.75) -- (-0.5, 1.25) -- (-1, 1.25);
	\draw[thick, red] (1, 2.25)  -- (0, 2.25)  -- (0, .75) -- (-1, 0.75);
	\draw[thick, black] (1, 0.25) -- (0.5, 0.25) -- (.5, 1.75) -- (1, 1.75);
	\draw[fill=black] (-1, 2.75) circle (0.075);
	\node[left] at (-1, 2.75) {$1$};
	\draw[fill=black] (1, 2.25) circle (0.075);
	\node[right] at (1, 2.25) {$2$};
	\draw[fill=black] (1, 1.75) circle (0.075);
	\node[right] at (1, 1.75) {$3$};
	\draw[fill=black] (-1, 1.25) circle (0.075);
	\node[left] at (-1, 1.25) {$4$};
	\draw[fill=black] (-1, .75) circle (0.075);
	\node[left] at (-1, .75) {$5$};
	\draw[fill=black] (1, 0.25) circle (0.075);
	\node[right] at (1, 0.25) {$6$};
	\node[below] at (0,0) {not vertically split};
\end{tikzpicture}
\qquad
\begin{tikzpicture}[baseline]
	\draw[thick, dashed] (-1,3) -- (-1,0) -- (1, 0) -- (1,3);
	\draw[dotted, thick] (0,0) -- (0, 3);
	\draw[thick, black] (-1, 2.75) -- (-0.5, 2.75) -- (-0.5, 1.25) -- (-1, 1.25);
	\draw[thick, black] (-0.5, 1.25)  -- (-0.5, .75) -- (-1, 0.75);
	\draw[thick, black] (1, 0.25) -- (0.5, 0.25) -- (.5, 2.25) -- (1, 2.25);
	\draw[fill=black] (-1, 2.75) circle (0.075);
	\node[left] at (-1, 2.75) {$1$};
	\draw[fill=black] (1, 2.25) circle (0.075);
	\node[right] at (1, 2.25) {$2$};
	\draw[fill=black] (1, 1.75) circle (0.075);
	\node[right] at (1, 1.75) {$3$};
	\draw[fill=black] (-1, 1.25) circle (0.075);
	\node[left] at (-1, 1.25) {$4$};
	\draw[fill=black] (-1, .75) circle (0.075);
	\node[left] at (-1, .75) {$5$};
	\draw[fill=black] (1, 0.25) circle (0.075);
	\node[right] at (1, 0.25) {$6$};
	\node[below] at (0,0) {vertically split};
\end{tikzpicture}
\end{align*}

\begin{rem}\label{rem:vertically-split}
Since we will always be in the context of Lemma \ref{lem:getting-the-bi-free-we-need} where all left operators are bi-free from all right operators, Theorem \ref{thm:mixed-cumulants-vanish} implies that $\kappa_\pi(Z_1, \ldots, Z_n) = 0$ whenever $\pi \in \BNC(\chi)$ is not vertically split.  
\end{rem}

Moreover, we will only need to consider $\chi$ of the form $\chi_{n,a} : [2n] \to \{\ell, r\}$ where
\[
\chi_{n,a}(k) = \begin{cases}
\ell & \text{if $k$ is odd} \\
r & \text{if $k$ is even}
\end{cases}.
\]
That is, $\chi_{n,a}$ alternates between left and right operators.  We will use $\BNC^a_{\vs}(n)$ in place of $\BNC_\vs(\chi_{n,a})$, and $\BNC^a_{\vs, 2}(n)$ for all elements $\pi \in \BNC^a_{\vs}(n)$ where every block of $\pi$ contains precisely two elements (i.e. the alternating, vertically split, pair bi-non-crossing partitions).

\section{Asymptotics of Tensors of Random Matrices via Bi-Free Matrix Models}
\label{sec:matrix-models}

In this section, we will use the bi-free matrix models from \cite{S2017} to obtain an analogous result to the fixed $d$ portion of \cite{LSY2023}*{Theorem 2.2}.  Instead of using the operator approach to bi-free matrix models from \cite{S2017}, we will use the equivalent tensor characterization from \cite{CS2021}*{Section 9}.  

Let $\mathcal{L}_{\overline{\infty}}(\mu)$ denote the collection of all random variables that have moments of all order with respect to a probability measure $\mu$.  Consider the algebra
\[
\A = \mathcal{L}_{\overline{\infty}}(\mu) \otimes M_n(\bC) \otimes M_n(\bC)^{\op}
\]
where $M_n(\bC)^{\op}$ denotes $M_n(\bC)$ with the opposite multiplication.  Further define a linear ``expectation" $E : \A \to \bC$ by 
\[
E(f \otimes T \otimes S) = \mathbb{E}(f) \tr(TS)
\]
where $\tr$ denotes the normalized trace on $M_n(\bC)$.  Note $E$ is not positive.

Let $L, R : \mathcal{L}_{\overline{\infty}}(\mu) \otimes M_n(\bC) \to \A$ be the linear maps such that
\[
L(f \otimes T) = f \otimes T \otimes I_n \qqand R(f \otimes T) = f \otimes I_n \otimes T.
\]
For self-adjoint $\{X_j\}^d_{j=1} \subseteq \mathcal{L}_{\overline{\infty}}(\mu) \otimes M_n(\bC)$ (i.e. our usual notion of random matrices), it is elementary to verify that the joint distribution of $\{L(X_j)\}^d_{j=1}$ with respect to $E$ is the joint distribution of $\{X_j\}^d_{j=1}$ with respect to $\tr \circ \mathbb{E}$.  Moreover, the joint distribution of $\{R(X_j)\}^d_{j=1}$ with respect to $E$ is the joint distribution of $\{\overline{X_j}\}^d_{j=1}$ with respect to $\tr \circ \mathbb{E}$ as
\begin{align*}
E(R(X_{j_1}) \cdots R(X_{j_k})) &= \mathbb{E}\left(\tr(X_{j_k} \cdots X_{j_1})\right) \\
&= \mathbb{E}\left(\tr\left((X_{j_k} \cdots X_{j_1})^t\right)\right) \\
&= \mathbb{E}\left(\tr\left(X_{j_1}^t \cdots X_{j_k}^t\right)\right) \\
&= \mathbb{E}\left(\tr\left(\overline{X_{j_1}} \cdots \overline{X_{j_k}}\right)\right).
\end{align*}
Note for all $X, Y \in \mathcal{L}_{\overline{\infty}}(\mu) \otimes M_n(\bC)$ that $L(X)R(Y) = R(Y) L(X)$.

The above plus \cite{S2017}*{Theorem 4.13} is all we need for the bi-free proof of our analogue of \cite{LSY2023}*{Theorem 2.2}.

\begin{thm}\label{thm:bi-matrix}
Let $W_1, \ldots, W_{2d}$ be self-adjoint random $n \times n$ matrices such that
\begin{itemize}
\item the distribution of $W_j$ converges to $\mu_j$ as $n$ tends to infinity for all $j$, and 
\item $\{W_j\}^{2d}_{j=1}$ are asymptotically free in expectation as $n$ tends to infinity.
\end{itemize}
Let $a_1, \ldots, a_{2d}$ be freely independent random variables with respect to $\varphi$ with distributions $\mu_1, \ldots, \mu_{2d}$ respectively and let
\[
\Delta_{d,n} = \frac{1}{\sqrt{d}} \sum^d_{j=1}\left(W_j \otimes \overline{W_{j+d}} -\mathbb{E}\left(\tr(W_j)I_n \otimes \tr(\overline{W_{j+d}})I_n\right)\right).
\]
Then the distribution of $\Delta_{d,n}$ tends to the distribution of 
\[
\frac{1}{\sqrt{d}} \sum^d_{j=1} a_j \otimes a_{j+d} - \varphi(a_j)1 \otimes \varphi(a_{j+d})1
\]
with respect to $\varphi \otimes \varphi$ as $n$ tends to infinity.  
\end{thm}
\begin{proof}
Using the above notation, since $\{L(W_j)\}^{2d}_{j=1}$ are asymptotically freely independent \cite{S2017}*{Theorem 4.13} implies that $\{(L(W_j), R(W_j))\}^{2d}_{j=1}$ are asymptotically bi-freely independent with respect to $E$.  Consequently
\[
\{(L(W_j), R(1)\}^{d}_{j=1} \cup \{(L(1), R(W_j))\}^{2d}_{j=d+1}
\]
are asymptotically bi-freely independent with respect to $E$. Thus for $\{i_1, \ldots, i_k\} \in \{1,\ldots, d\}$ and $\{j_1, \ldots, j_m\} \in \{d+1, \ldots, 2d\}$, we have that
\begin{align*}
\lim_{n \to\infty} & E(L(W_{i_1}) \cdots L(W_{i_k}) R(W_{j_1}) \cdots R(W_{j_m})) \\
& = \lim_{n \to\infty} E(L(W_{i_1}) \cdots L(W_{i_k})) E(R(W_{j_1}) \cdots R(W_{j_m})) \\
&=  \lim_{n \to\infty} \tr(\mathbb{E}(W_{i_1} \cdots W_{i_k}))\tr(\mathbb{E}(\overline{W_{j_1}} \cdots \overline{W_{j_m}})).
\end{align*}
Consequently, since $\lambda_j = \mathbb{E}(\tr(W_j))$ and $\rho_j = \mathbb{E}(\tr(\overline{W_j}))$ are scalars (dependent on $n$), if
\[
\Gamma_{d,n} = \frac{1}{\sqrt{d}} \sum^d_{j=1}\left(L(W_j) R(W_{j+d}) - L(\lambda_j I_n) R(\rho_{j+d} I_n))\right).
\]
we obtain for all $m$ that
\[
\lim_{n \to \infty} \mathbb{E}(\tr(\Delta_{d,n}^m)) = \lim_{n \to \infty} E(\Gamma_{d,n}^m).
\]

Since the distribution of $W_j$ converges to the distribution of $a_j$ for all $j$, and since $\{(L(W_j), R(1)\}^{d}_{j=1} \cup \{(L(1), R(W_j))\}^{2d}_{j=d+1}$ are asymptotically bi-freely independent, we obtain by Lemma \ref{lem:getting-the-bi-free-we-need} that the distribution of $\Gamma_{d,n}$ tends to the distribution of
\[
\frac{1}{\sqrt{d}} \sum^d_{j=1} a_j \otimes a_{j+d} - \varphi(a_j) \otimes \varphi(a_{j+d})
\]
with respect to $\varphi \otimes \varphi$ as $n$ tends to infinity.
\end{proof}

Although Theorem \ref{thm:bi-matrix} simplifies the operator $\Delta_{d,n}$ from \cite{LSY2023}*{Theorem 2.2} so that the term subtracted off is just a scalar multiple of the identity and only obtains the asymptotic behaviour of $\Delta_{d,n}$ in expectation opposed to in probability and expectation, Theorem \ref{thm:bi-matrix} does drop the centred condition and allows for the two sides of each tensor to have different random matrices.  Provided $W_j$ and $W_{j+d}$ have the same asymptotic distribution, one can replace $W_j \otimes \overline{W_{j+d}}$ with $W_j \otimes \overline{W_{j}}$.  Moreover, the asymptotic distributions of $\Delta_{d,n}$ as $d$ and $n$ tend to infinity are characterized by the Central Limit Theorem from \cite{LSY2024}*{Theorem 1.1} that will be demonstrated using bi-free techniques in Section \ref{sec:general}.

\section{CLT of Tensors of Centred Freely Independent Operators}
\label{sec:centred}

In this section, we will provide a short proof of \cite{LSY2024}*{Theorem 1.1} in the case of centred operators using bi-free techniques.

\begin{thm}\label{thm:centred-result}
Let $(\A, \varphi)$ be a non-commutative probability space and let $a, b \in \A$ be such that $\varphi(a) = \varphi(b) = 0$.  Let $(a_k)_{k\geq 1}$ be a sequence of freely independent copies of $a$ in $(\A, \varphi)$  and let $(b_k)_{k\geq 1}$ be a sequence of freely independent copies of $b$ in $(\A, \varphi)$.  Let
\[
S_n = \frac{1}{\sqrt{n}} \sum^n_{k=1} a_k \otimes b_k \in \A \otimes \A.
\]
Then for all $m \in \bN$,
\[
\lim_{n \to \infty} (\varphi \otimes \varphi)(S_n^m) = \begin{cases}
0 & \text{if $m$ is odd} \\
|NC_2(m)| \varphi(a^2)^\frac{m}{2}\varphi(b^2)^\frac{m}{2} & \text{if $m$ is even}
\end{cases}.
\]
Thus, if $a$ and $b$ are self-adjoint elements with variance 1, the above limit is the semicircular distribution with variance 1.
\end{thm}

Before proceeding with the bi-free proof of Theorem \ref{thm:centred-result}, we desire to introduce some terminology that, although can be bypassed in the proof, we feel is enlightening and draws connections to other work and problems.  

\begin{defn}
A meandric system of size $m$ is a picture obtained by drawing $2m$ points along a horizontal line and drawing continuous, curved, non-self-intersecting, looped paths that passes through the line at all of the $2m$ points exactly once.   

Given a meandric system $M$, we will use $c(M)$ to denote the number of closed loops contained in $M$.
\end{defn}
The following figure is an example of a meandric system of size $4$ with 2 closed loops.
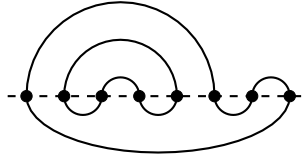
\begin{align*}
\begin{tikzpicture}[baseline]
	\draw[thick, dashed] (-0.25,1) --(3.75,1);
	\draw[fill=black] (0, 1) circle (0.075);
	\draw[fill=black] (0.5, 1) circle (0.075);
	\draw[fill=black] (1, 1) circle (0.075);
	\draw[fill=black] (1.5, 1) circle (0.075);
	\draw[fill=black] (2, 1) circle (0.075);
	\draw[fill=black] (2.5, 1) circle (0.075);
	\draw[fill=black] (3, 1) circle (0.075);
	\draw[fill=black] (3.5, 1) circle (0.075);
	\draw [thick,domain=0:180] plot ({1.25*cos(\x)+1.25}, {1.25*sin(\x)+1});
	\draw [thick,domain=0:180] plot ({0.75*cos(\x)+1.25}, {0.75*sin(\x)+1});
	\draw [thick,domain=0:180] plot ({0.25*cos(\x)+1.25}, {0.25*sin(\x)+1});
	\draw [thick,domain=0:180] plot ({0.25*cos(\x)+3.25}, {0.25*sin(\x)+1});
	\draw [thick,domain=180:360] plot ({1.75*cos(\x)+1.75}, {.75*sin(\x)+1});
	\draw [thick,domain=180:360] plot ({.25*cos(\x)+.75}, {0.25*sin(\x)+1});
	\draw [thick,domain=180:360] plot ({.25*cos(\x)+1.75}, {0.25*sin(\x)+1});
	\draw [thick,domain=180:360] plot ({.25*cos(\x)+2.75}, {0.25*sin(\x)+1});
	\end{tikzpicture}
\end{align*}

The meandric systems of size $m$ are in bijective correspondence between pairs of elements of $\NC_2(2m)$ where a meandric system $M$ corresponds to the pair $(\pi, \sigma)$ where $\{i,j\}$ is a block of $\pi$ if and only if $i$ and $j$ are connected via a solid line above the horizontal line and $\{i,j\}$ is a block of $\sigma$ if and only if $i$ and $j$ are connected via a solid line below the horizontal line.  Note $c(M) = m$ exactly when $\pi = \sigma$.

Since $\BNC^a_{\vs}(2m)$ is in bijective correspondence with pairs of elements of $\NC_2(2m)$ as $\pi \in \BNC^a_{\vs}(2m)$ is uniquely defined by an element of $\NC_2(2m)$ on $\{1,3, \ldots, 2m-1\}$ paired with an element of $\NC_2(2m)$ on $\{2, 4, \ldots, 2m\}$, there is a bijection $M$ from $\BNC^a_{\vs}(2m)$ to the set of all meandric systems of size $m$.  This bijection is demonstrated in the following figure.
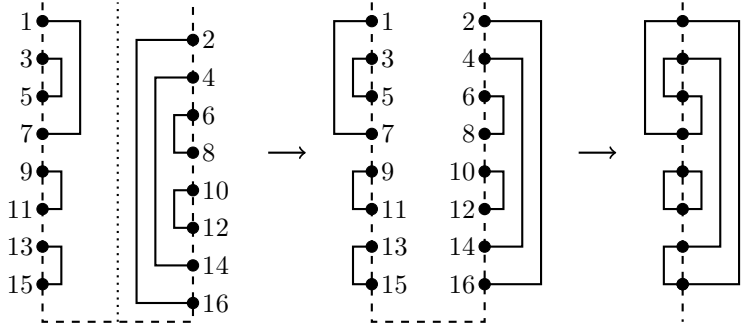
\begin{align*}
\begin{tikzpicture}[baseline]
	\draw[thick, dashed] (-1,4) -- (-1,-.25) -- (1, -.25) -- (1,4);
	\draw[dotted, thick] (0,-.25) -- (0, 4);
	\draw[fill=black] (-1, 0.25) circle (0.075);
	\node[left] at (-1, 0.25) {$15$};
	\draw[fill=black] (-1, 0.75) circle (0.075);
	\node[left] at (-1, 0.75) {$13$};	
	\draw[fill=black] (-1, 1.25) circle (0.075);
	\node[left] at (-1, 1.25) {$11$};
	\draw[fill=black] (-1, 1.75) circle (0.075);
	\node[left] at (-1, 1.75) {$9$};	
	\draw[fill=black] (-1, 2.25) circle (0.075);
	\node[left] at (-1, 2.25) {$7$};
	\draw[fill=black] (-1, 2.75) circle (0.075);
	\node[left] at (-1, 2.75) {$5$};	
	\draw[fill=black] (-1, 3.25) circle (0.075);
	\node[left] at (-1, 3.25) {$3$};
	\draw[fill=black] (-1, 3.75) circle (0.075);
	\node[left] at (-1, 3.75) {$1$};	
	\draw[fill=black] (1, 0.0) circle (0.075);
	\node[right] at (1, 0.0) {$16$};
	\draw[fill=black] (1, 0.5) circle (0.075);
	\node[right] at (1, 0.5) {$14$};	
	\draw[fill=black] (1, 1.0) circle (0.075);
	\node[right] at (1, 1.0) {$12$};
	\draw[fill=black] (1, 1.5) circle (0.075);
	\node[right] at (1, 1.5) {$10$};	
	\draw[fill=black] (1, 2.0) circle (0.075);
	\node[right] at (1, 2.0) {$8$};
	\draw[fill=black] (1, 2.5) circle (0.075);
	\node[right] at (1, 2.5) {$6$};	
	\draw[fill=black] (1, 3.0) circle (0.075);
	\node[right] at (1, 3.0) {$4$};
	\draw[fill=black] (1, 3.5) circle (0.075);
	\node[right] at (1, 3.5) {$2$};	
	\draw[thick, black] (-1, 3.75) -- (-0.5, 3.75) -- (-0.5, 2.25) -- (-1, 2.25);
	\draw[thick, black] (-1, 3.25) -- (-0.75, 3.25) -- (-0.75, 2.75) -- (-1, 2.75) ;
	\draw[thick, black] (-1, 1.75) -- (-0.75, 1.75) -- (-0.75, 1.25)  -- (-1, 1.25) ;
	\draw[thick, black] (-1, 0.25) -- (-0.75, 0.25) -- (-0.75, 0.75) -- (-1, 0.75) ;
	\draw[thick, black] (1, 3.5)  -- (0.25, 3.5)  -- (0.25, 0.0) -- (1, 0.0) ;
	\draw[thick, black] (1, 3.0) -- (0.5, 3.0) -- (0.5, 0.5) -- (1, 0.5);
	\draw[thick, black] (1, 2.5) -- (0.75, 2.5) -- (0.75, 2.0) -- (1, 2.0) ;
	\draw[thick, black] (1, 1.5) -- (0.75, 1.5) -- (0.75, 1.0)  -- (1, 1.0)  ;
	\draw[thick, ->] (2, 2) -- (2.5,2);
\end{tikzpicture}
\quad
\begin{tikzpicture}[baseline]
	\draw[thick, dashed] (-0.75,4) -- (-0.75,-.25) -- (0.75, -.25) -- (0.75,4);
	\draw[fill=black] (-0.75, 0.25) circle (0.075);
	\node[right] at (-0.75, 0.25) {$15$};
	\draw[fill=black] (-0.75, 0.75) circle (0.075);
	\node[right] at (-0.75, 0.75) {$13$};	
	\draw[fill=black] (-0.75, 1.25) circle (0.075);
	\node[right] at (-0.75, 1.25) {$11$};
	\draw[fill=black] (-0.75, 1.75) circle (0.075);
	\node[right] at (-0.75, 1.75) {$9$};	
	\draw[fill=black] (-0.75, 2.25) circle (0.075);
	\node[right] at (-0.75, 2.25) {$7$};
	\draw[fill=black] (-0.75, 2.75) circle (0.075);
	\node[right] at (-0.75, 2.75) {$5$};	
	\draw[fill=black] (-0.75, 3.25) circle (0.075);
	\node[right] at (-0.75, 3.25) {$3$};
	\draw[fill=black] (-0.75, 3.75) circle (0.075);
	\node[right] at (-0.75, 3.75) {$1$};	
	\draw[fill=black] (0.75, 0.25) circle (0.075);
	\node[left] at (0.75, 0.25) {$16$};
	\draw[fill=black] (0.75, 0.75) circle (0.075);
	\node[left] at (0.75, 0.75) {$14$};	
	\draw[fill=black] (0.75, 1.25) circle (0.075);
	\node[left] at (0.75, 1.25) {$12$};
	\draw[fill=black] (0.75, 1.75) circle (0.075);
	\node[left] at (0.75, 1.75) {$10$};	
	\draw[fill=black] (0.75, 2.25) circle (0.075);
	\node[left] at (0.75, 2.25) {$8$};
	\draw[fill=black] (0.75, 2.75) circle (0.075);
	\node[left] at (0.75, 2.75) {$6$};	
	\draw[fill=black] (0.75, 3.25) circle (0.075);
	\node[left] at (0.75, 3.25) {$4$};
	\draw[fill=black] (0.75, 3.75) circle (0.075);
	\node[left] at (0.75, 3.75) {$2$};	
	\draw[thick, black] (-0.75, 3.75) -- (-1.25, 3.75) -- (-1.25, 2.25) -- (-0.75, 2.25);
	\draw[thick, black] (-0.75, 3.25) -- (-1.0, 3.25) -- (-1.0, 2.75) -- (-0.75, 2.75) ;
	\draw[thick, black] (-0.75, 1.75) -- (-1.0, 1.75) -- (-1.0, 1.25)  -- (-0.75, 1.25) ;
	\draw[thick, black] (-0.75, 0.25) -- (-1.0, 0.25) -- (-1.0, 0.75) -- (-0.75, 0.75) ;
	\draw[thick, black] (0.75, 3.75)  -- (1.5, 3.75)  -- (1.5, 0.25) -- (0.75, 0.25) ;
	\draw[thick, black] (0.75, 3.25) -- (1.25, 3.25) -- (1.25, 0.75) -- (0.75, 0.75);
	\draw[thick, black] (0.75, 2.75) -- (1.0, 2.75) -- (1.0, 2.25) -- (0.75, 2.25) ;
	\draw[thick, black] (0.75, 1.75) -- (1.0, 1.75) -- (1.0, 1.25)  -- (0.75, 1.25)  ;
	\draw[thick, ->] (2, 2) -- (2.5,2);
\end{tikzpicture}
\quad
\begin{tikzpicture}[baseline]
	\draw[thick, dashed] (0,4) -- (0,-.25);
	\draw[fill=black] (0, 0.25) circle (0.075);
	\draw[fill=black] (0, 0.75) circle (0.075);
	\draw[fill=black] (0, 1.25) circle (0.075);
	\draw[fill=black] (0, 1.75) circle (0.075);
	\draw[fill=black] (0, 2.25) circle (0.075);
	\draw[fill=black] (0, 2.75) circle (0.075);
	\draw[fill=black] (0, 3.25) circle (0.075);
	\draw[fill=black] (0, 3.75) circle (0.075);
	\draw[fill=black] (0, 0.25) circle (0.075);
	\draw[thick, black] (0, 3.75) -- (-.5, 3.75) -- (-.5, 2.25) -- (0, 2.25);
	\draw[thick, black] (0, 3.25) -- (-0.25, 3.25) -- (-0.25, 2.75) -- (0, 2.75) ;
	\draw[thick, black] (0, 1.75) -- (-0.25, 1.75) -- (-0.25, 1.25)  -- (0, 1.25) ;
	\draw[thick, black] (0, 0.25) -- (-0.25, 0.25) -- (-0.25, 0.75) -- (0, 0.75) ;
	\draw[thick, black] (0, 3.75)  -- (0.75, 3.75)  -- (0.75, 0.25) -- (0, 0.25) ;
	\draw[thick, black] (0, 3.25) -- (.5, 3.25) -- (.5, 0.75) -- (0, 0.75);
	\draw[thick, black] (0, 2.75) -- (0.25, 2.75) -- (0.25, 2.25) -- (0, 2.25) ;
	\draw[thick, black] (0, 1.75) -- (0.25, 1.75) -- (0.25, 1.25)  -- (0, 1.25)  ;
\end{tikzpicture}
\end{align*}

Note \cites{N2016, NZ2018} demonstrated a connection between meandric systems and sums of products of left and right variable.  The following proof further demonstrates the connection although it eliminates all but the simplest meandric systems through asymptotics.

\begin{proof}[Proof of Theorem \ref{thm:centred-result}.]
For all $k$ let
\[
A_k = a_k \otimes 1 \qqand B_k = 1 \otimes b_k.
\]
Therefore
\[
\{(A_k, 1)\}^n_{k=1} \cup \{(1, B_k)\}^n_{k=1}
\]
is a bi-free collection with respect to $\varphi \otimes \varphi$ by Lemma \ref{lem:getting-the-bi-free-we-need}.  Hence, by Remark \ref{rem:vertically-split}, we obtain 
\begin{align*}
(\varphi \otimes \varphi)(S_n^m) &= \frac{1}{n^{\frac{m}{2}}} \sum_{\theta : [m] \to [n]}  (\varphi \otimes \varphi)\left(\prod_{k=1}^m A_{\theta(k)}B_{\theta(k)} \right) \\
&= \frac{1}{n^{\frac{m}{2}}} \sum_{\theta : [m] \to [n]} \sum_{\pi \in \BNC^a_{\vs}(m)} \kappa_{\pi}(A_{\theta(1)}, B_{\theta(1)}, \ldots, A_{\theta(m)}, B_{\theta(m)})\\
&= \frac{1}{n^{\frac{m}{2}}} \sum_{\pi \in \BNC^a_{\vs}(m)}  \sum_{\theta : [m] \to [n]} \kappa_{\pi}(A_{\theta(1)}, B_{\theta(1)}, \ldots, A_{\theta(m)}, B_{\theta(m)}).
\end{align*}

Given $\theta : [m] \to [n]$, define $\theta' : [2m] \to [n]$ by
\[
\theta'(k) = \begin{cases}
\theta\left(\frac{k}{2}\right) & \text{if $k$ is even} \\
\theta\left(\frac{k+1}{2}\right) & \text{if $k$ is odd} \\
\end{cases}.
\]
Note $\theta'$ defines a partition on $[2m]$ with blocks $(\theta')^{-1}(\{k\})$ for all $k \in [n]$.

Since mixed cumulants vanish, for all $\pi \in \BNC^a_{\vs}(m)$ and $\theta : [m] \to [n]$ we obtain that $\kappa_{\pi}(A_{\theta(1)}, B_{\theta(1)}, \ldots, A_{\theta(m)}, B_{\theta(m)}) = 0$ unless $\pi \leq \theta'$ (i.e. we must be able to `colour' $\pi$ based on $\theta'$).  Moreover, since $(A_k)_{k\geq 1}$ and $(B_k)_{k\geq 1}$ are both identically distributed, there exists a constant $\kappa_\pi$ such that if $\pi \leq\theta'$ then $\kappa_{\pi}(A_{\theta(1)}, B_{\theta(1)}, \ldots, A_{\theta(m)}, B_{\theta(m)}) = \kappa_\pi$.  Hence
\begin{align*}
(\varphi \otimes \varphi)(S_n^m) &= \frac{1}{n^{\frac{m}{2}}} \sum_{\pi \in \BNC^a_{\vs}(m)}  \left| \left\{ \theta : [m] \to [n] \, \mid \,  \pi \leq \theta' \right\}\right|  \kappa_\pi.
\end{align*}

Note since $a$ and $b$ are centred that $\kappa_\pi = 0$ if $\pi$ has a block of with exactly one element. Moreover, if $\pi$ does not have a block with just one element and has a block with at least 3 elements, then
\[
\left| \left\{ \theta : [m] \to [n] \, \mid \,  \pi \leq \theta' \right\}\right| \leq n^{1 + \frac{m-3}{2}} = n^{\frac{m-1}{2}}.
\]
Therefore, since the sum is finite, the above $\pi \in \BNC^a_{\vs}(m)$ contribute 0 to the limit as $n$ tends to infinity.  Hence 
\[
\lim_{n \to \infty} (\varphi \otimes \varphi)(S_n^m) = \lim_{n \to \infty} \frac{1}{n^{\frac{m}{2}}} \sum_{\pi \in \BNC^a_{\vs,2}(m)}  \left| \left\{ \theta : [m] \to [n] \, \mid \,  \pi \leq \theta' \right\}\right|  \kappa_\pi
\]
provide the limit on the right exists.  Thus $\lim_{n \to \infty} (\varphi \otimes \varphi)(S_n^m) = 0$ when $m$ is odd.

Assume $m$ is even.  For $\pi \in \BNC^a_{\vs,2}(m)$, consider the meandric system $M(\pi)$ corresponding to $\pi$ of size $\frac{m}{2}$.  Note if $k$ and $\ell$ are in the same loop of $M(\pi)$ and $\theta : [m] \to [n]$ is such that $\pi \leq \theta'$, then $\theta(k) = \theta(\ell)$.  Thus 
\[
\left| \left\{ \theta : [m] \to [n] \, \mid \, \pi \leq \theta' \right\}\right| = n^{c(M(\pi))}.
\]
Since $M(\pi)$ can have at most $\frac{m}{2}$ connected components, we see that
\[
\lim_{n \to\infty} \frac{1}{n^{\frac{m}{2}}} \left| \left\{ \theta : [m] \to [n] \, \mid \,  \pi \leq \theta' \right\}\right|   = \begin{cases}
1 & \text{if }c(M(\pi)) = \frac{m}{2} \\
0 & \text{if } c(M(\pi)) \neq \frac{m}{2}
\end{cases}.
\]
Hence
\[
\lim_{n \to \infty} (\varphi \otimes \varphi)(S_n^m) = \sum_{\substack{\pi \in \BNC^a_{\vs,2}(m) \\ c(M(\pi)) = \frac{m}{2} }   } \kappa_\pi.
\]
By previous discussions, the set of all $\pi \in \BNC^a_{\vs,2}(m)$ with $c(M(\pi)) = \frac{m}{2}$ is in bijective correspondence with  $\NC_2(m)$ and 
\[
\kappa_\pi = \kappa_{1_2}(a,a)^\frac{m}{2} \kappa_{1_2}(b,b)^\frac{m}{2} =  \varphi(a^2)^\frac{m}{2}\varphi(b^2)^\frac{m}{2} 
\]
for all such $\pi$ thereby completing the proof.
\end{proof}

\section{CLT of Tensors of Freely Independent Operators}
\label{sec:general}

In this section, a proof of \cite{LSY2024}*{Theorem 1.1} using bi-free techniques will be provided.  To begin, some notation is required.

Let $\mu_{sc}$ denote the distribution of the centred semicircular operator with variance 1.  Given a self-adjoint random variable $a$ with distribution $\mu$, for $t \in (0, 1)$ the distribution of $ta$ is denoted $t\mu$ and can computed via the dilation of $\mu$ by $t$: $(t\mu)(A) = \mu(t^{-1}A)$ for all Borel subsets $A \subseteq \bR$.  For $q \in [0,1)$, let 
\[
\mu_q = \sqrt{q}\left(\frac{1}{\sqrt{2}} \mu_{sc} \oplus \frac{1}{\sqrt{2}} \mu_{sc}\right) \boxplus \sqrt{1-q} \mu_{sc}
\]
where $\oplus$ denotes the distribution obtain by adding classically independent copies of the random variables and $\boxplus$ denotes the distribution obtained by adding freely independent copies of the random variables (i.e. the free additive convolution).  The goal of this section is to show that $\mu_q$ is the asymptotic distribution of averaging tensors of freely independent random variables.

\begin{thm}[\cite{LSY2024}*{Theorem 1.1}]\label{thm:main}
Let $(\A, \varphi)$ be a non-commutative probability space and let $a, b \in \A$ be such that $\varphi(a) = \varphi(b) = \lambda$ and $\mathrm{var}(a) = \mathrm{var}(b) = \sigma^2 \neq 0$. Let
\[
\delta^2 = \mathrm{var}(a\otimes a) = \mathrm{var}(b\otimes b) = \sigma^2(\sigma^2 + 2\lambda^2)
\] 
and
\[
q = \frac{2\lambda^2}{\sigma^2 + 2 \lambda^2} \in [0,1).
\]

Let $(a_k)_{k\geq 1}$ be a sequence of freely independent copies of $a$ in $(\A, \varphi)$ and let $(b_k)_{k\geq 1}$ be a sequence of freely independent copies of $b$ in $(\A, \varphi)$.  Let
\[
S_n = \frac{1}{\delta \sqrt{n}} \sum^n_{k=1} (a_k \otimes b_k - \lambda 1 \otimes \lambda 1) \in \A \otimes \A.
\]
Then the distribution of $(S_n)_{\geq 1}$ with respect to $\varphi \otimes \varphi$ converges in distribution to $\mu_q$.
\end{thm}

Before proceeding with the bi-free proof of Theorem \ref{thm:main}, we first require some basic information about $\mu_q$. Since the best approach to obtain said information is via classical and free independence, we simply refer to \cite{LSY2024} as bi-free independence can yield no new information.

Given a partition $\pi \in \P(n)$, two distinct blocks $V, W \in \pi$ are said to \emph{interest} (or cross) each other if there exists $v_1, v_2  \in V$ and $w_1, w_2 \in W$ such that $v_1 < w_1 < v_2 < w_2$.  The \emph{intersection graph of $\pi$} is the graph whose vertices are the blocks of $\pi$ where two blocks are connected via an edge if and only if they intersect.  Let $\P_2^{\mathrm{con}}(n)$ denote the set of all pair partitions on $[n]$ with connected intersection graph and let $\P_2^{\mathrm{bicon}}(n)$ denote the set of all pair partitions on $[n]$ whose intersection graph is bipartite and connected.  Note $1_2$ is the only element of $\P_2^{\mathrm{bicon}}(2)$.

\begin{lem}[\cite{LSY2024}*{Proposition 3.1}]
Let
\[
\mu_1 = \frac{1}{\sqrt{2}} \mu_{sc} \oplus \frac{1}{\sqrt{2}} \mu_{sc}
\]
where $\oplus$ denotes the distribution obtain by adding classically independent copies of the random variables.  If $\kappa_n(\mu_1)$ denotes the $n^{\mathrm{th}}$ free cumulant of $\mu_1$, then
\[
\kappa_n(\mu_1) = \begin{cases}
0 & \text{if $n$ is odd} \\
2\left(\frac{1}{2} \right)^{\frac{n}{2}} | \P_2^{\mathrm{bicon}}(n) | & \text{if $n$ is even}
\end{cases}.
\]
\end{lem}

Using elementary facts from the combinatorics of free probability (\cite{NS2006}), we obtain the following recursive description of the moments of $\mu_q$.

\begin{cor}\label{cor:recurrence-relation}
Let $M_n$ denote the $n^{\mathrm{th}}$ moment of $\mu_q$.  Then $M_n = 0$ if $n$ is odd, $M_2 = 1$, and
\[
M_n = \sum_{\substack{0 \leq k_1, k_2 \leq n-2 \\ k_1 + k_2 = n-2}} M_{k_1} M_{k_2} + \sum_{j \geq 2} \sum_{\substack{0 \leq k_1, \ldots, k_{2j} \leq n-2j \\ k_1 + \cdots + k_{2j} = n-2j}} 2\left(\frac{q}{2} \right)^{j} |\P_2^{\mathrm{bicon}}(2j) | M_{k_1} \cdots M_{k_{2j}}
\]
if $n$ is even and $n \geq 4$.
\end{cor}
\begin{proof}
Let $S$ and $T$ be operators with distributions $\mu_{sc}$ and $\mu_1$ respectively that are freely independent with respect to a state $\varphi$.  Hence if $Z = \sqrt{1-q}S + \sqrt{q}T$ then
\[
M_n = \varphi(Z^n)
\]
for all $n$.  Since the free cumulants are multilinear, since mixed free cumulants vanish, since $\kappa_2(S) = 1$, and since $\kappa_n(S) = 0$ for all $n \neq 2$, we obtain that
\begin{align*}
\kappa_n(Z) &= \left(\sqrt{1-q}\right)^n\kappa_n(S) + \sqrt{q}^n\kappa_n(T)\\
&= \begin{cases}
0 & \text{if $n$ is odd} \\
1 & \text{if $n=2$} \\
2\left(\frac{q}{2} \right)^{\frac{n}{2}} |\P_2^{\mathrm{bicon}}(n) | & \text{if $n$ is even and $n\geq 4$}
\end{cases}.
\end{align*}

Recall the bi-free moment-cumulant formula (\ref{eq:moment}) reduces to the free moment-cumulant formula
\[
M_n = \varphi(Z^n) = \sum_{\pi \in \NC(n)} \kappa_\pi(Z, Z, \ldots, Z).
\]
Note $\kappa_k(Z) = 0$ when $k$ is odd. Thus, if $n$ is odd then $M_n = 0$ since every $\pi \in \NC(n)$ must have a block with an odd number of elements.  Moreover
\[
M_2 = \kappa_2(Z) = 1.
\]

To obtain the recursive formula, we will now use the common trick from \cite{NS2006} of summing over all possible partitions with the same block containing 1 in the moment-cumulant formula for $M_n$.  When doing this sum, we end up with the free cumulant corresponding to the block containing 1 and the joint moment of the operators in each region isolated by this block.  We believe the following figure aids in the comprehension of this argument.

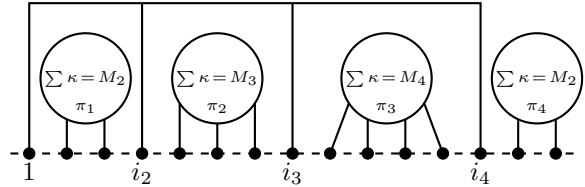
\begin{align*}
\begin{tikzpicture}[baseline]
	\draw[thick, dashed] (-.25,0) -- (7.25, 0);
	\draw[black, fill=black] (0,0) circle (0.075);
	\draw[black, fill=black] (0.5,0) circle (0.075);
	\draw[black, fill=black] (1,0) circle (0.075);
	\draw[black, fill=black] (1.5,0) circle (0.075);
	\draw[black, fill=black] (2,0) circle (0.075);
	\draw[black, fill=black] (2.5,0) circle (0.075);
	\draw[black, fill=black] (3,0) circle (0.075);
	\draw[black, fill=black] (3.5,0) circle (0.075);
	\draw[black, fill=black] (4,0) circle (0.075);
	\draw[black, fill=black] (4.5,0) circle (0.075);
	\draw[black, fill=black] (5,0) circle (0.075);
	\draw[black, fill=black] (5.5,0) circle (0.075);
	\draw[black, fill=black] (6,0) circle (0.075);
	\draw[black, fill=black] (6.5,0) circle (0.075);
	\draw[black, fill=black] (7,0) circle (0.075);
	\node[below] at (0,0) {$1$};
	\node[below] at (1.5,0) {$i_2$};
	\node[below] at (3.5,0) {$i_3$};
	\node[below] at (6,0) {$i_4$};
	\draw[black, thick] (0,0) -- (0, 2) -- (6, 2) -- (6,0);
	\draw[black, thick] (1.5,0) -- (1.5, 2);
	\draw[black, thick] (3.5,0) -- (3.5, 2);
	\node at (0.75, 1) {\tiny$\sum \kappa \! =\! M_2$};
	\node at (0.75, 0.6) {\tiny$\pi_1$};
	\node at (2.5, 1) {\tiny$\sum \kappa \! =\! M_3$};
	\node at (2.5, 0.6) {\tiny$\pi_2$};
	\node at (4.75, 1) {\tiny$\sum \kappa \! =\! M_4$};
	\node at (4.75, 0.6) {\tiny$\pi_3$};
	\node at (6.75, 1) {\tiny$\sum \kappa \! =\! M_2$};
	\node at (6.75, 0.6) {\tiny$\pi_4$};
	\draw[thick] (0.75,1) ellipse (.6 and .6);
	\draw[thick] (2.5,1) ellipse (.6 and .6);
	\draw[thick] (4.75,1) ellipse (.6 and .6);
	\draw[thick] (6.75,1) ellipse (.6 and .6);
	\draw[thick] (0.5, 0) -- (0.5, .45);
	\draw[thick] (1, 0) -- (1, .45);
	\draw[thick] (6.5, 0) -- (6.5, .45);
	\draw[thick] (7, 0) -- (7, .45);
	\draw[thick] (2.5, 0) -- (2.5, .4);
	\draw[thick] (2, 0) -- (2, .67);
	\draw[thick] (3, 0) -- (3, .67);
	\draw[thick] (4.5, 0) -- (4.5, .45);
	\draw[thick] (5, 0) -- (5, .45);
	\draw[thick] (4, 0) -- (4.25, .67);
	\draw[thick] (5.5, 0) -- (5.25, .67);
\end{tikzpicture}
\end{align*}

By our knowledge of the free cumulants of $Z$, we know that $\kappa_\pi(Z, Z, \ldots, Z)  = 0$ unless the block of $\pi$ containing $1$ has an even number of elements.  For such a $\pi \in \NC(n)$, let $1 = i_1 < i_2 < \cdots < i_{2j}$ be the elements of the block of $\pi$ containing 1 for some $j \geq 1$.  There is a bijection from such $\pi$ to $2j$-tuples $(\pi_1, \ldots, \pi_{2j})$ where $\pi_k= \pi|_{\{i_k+1, \ldots, i_{k+1}-1\}}$ is non-crossing on $\{i_k+1, \ldots, i_{k+1}-1\}$ for $1 \leq k \leq 2j$ where $i_{2j+1} = n+1$.  Let $k_h = i_{h+1} - i_h -1$ for $1 \leq h \leq 2j$.  Then $0 \leq k_1,\ldots, k_{2j} \leq n-2j$ and $k_1 + \cdots + k_{2j} = n-2j$.  Since for such a $\pi$ we have
\[
\kappa_\pi(Z, Z, \ldots, Z) = \kappa_{2j}(Z) \kappa_{\pi_1}(Z, \ldots, Z)\cdots \kappa_{\pi_{2j}}(Z, \ldots, Z),
\]
we obtain $\kappa_{2j}(Z) M_{k_1} \cdots M_{k_{2j}}$ by summing $\kappa_\pi(Z, Z, \ldots, Z)$ over all $\pi \in \NC(n)$ with $\{1 < i_1 < \cdots < i_{2j-1}\}$ a block of $\pi$.  Since for every $j \geq 1$ and $0 \leq k_1,\ldots, k_{2j} \leq n-2j$ such that $k_1 + \cdots + k_{2j} = n-2j$ there is a selection of $1 = i_1 < i_2 < \cdots < i_{2j}$ that yield $k_1, \ldots, k_{2j}$, the result follows.
\end{proof}

\begin{proof}[Proof of Theorem \ref{thm:main}.]
For all $k$ let
\[
A_k = a_k \otimes 1 \qqand B_k = 1 \otimes b_k.
\]
Therefore
\[
\{(A_k, 1)\}^n_{k=1} \cup \{(1, B_k)\}^n_{k=1}
\]
is a bi-free collection with respect to $\varphi \otimes \varphi$ by Lemma \ref{lem:getting-the-bi-free-we-need}.  For all $k$ let
\[
Z_k = A_kB_k - \lambda^2 (1 \otimes 1) \in \A \otimes \A.
\]
Hence
\[
(\varphi \otimes \varphi)(S^m_n) = \frac{1}{\delta^m n^{\frac{m}{2}}} \sum_{\theta : [m] \to [n]} (\varphi \otimes \varphi) \left(\prod_{k=1}^m Z_{\theta(k)}\right).
\]

Since $Z_1, \ldots, Z_n$ are identically distributed and satisfy the same independence relations, the value of
\[
(\varphi \otimes \varphi) \left(\prod_{k=1}^m Z_{\theta(k)}\right)
\]
depends only on the partition $\pi = \pi(\theta) \in \P(m)$ whose blocks are $\theta^{-1}(\{k\})$ for $k \in [n]$ and will be denoted $(\varphi \otimes \varphi)(\pi)$.  Thus
\[
(\varphi \otimes \varphi)(S^m_n) = \frac{1}{\delta^m n^{\frac{m}{2}}} \sum_{\pi \in \P(m)} (\varphi \otimes \varphi)(\pi) |\{\theta : [m] \to [n] \, \mid \, \pi = \pi(\theta)\}|.
\]  

For a given $\pi \in \P(m)$, the number of $\theta : [m] \to [n]$ such that $\pi(\theta) = \pi$ is equal to the number of ways we can label (i.e. `colour') the blocks of $\pi$ with distinct elements of $[n]$ and thus is equal to $n(n-1) \cdots (n- |\pi| + 1)$ where $|\pi|$ is the number of blocks of $\pi$.  Hence
\[
(\varphi \otimes \varphi)(S^m_n) = \frac{1}{\delta^m} \sum_{\pi \in \P(m)} (\varphi \otimes \varphi)(\pi) \frac{n(n-1) \cdots (n- |\pi| + 1)}{n^{\frac{m}{2}}}.
\]

We claim that $(\varphi \otimes \varphi)(\pi) = 0$ if $\pi$ has a block with exactly one element.  To see this, assume $\pi \in \P(m)$ has $\{k_0\}$ as a block for some $k_0$.  Note
\begin{align*}
(\varphi \otimes \varphi)(\pi) &= (\varphi \otimes \varphi)\left( Z_{\theta(1)} \cdots Z_{\theta(k_0-1)} A_{\theta(k_0)} B_{\theta(k_0)} Z_{\theta(k_0+1)} \cdots Z_{\theta(m)}\right) \\
& \quad -\lambda^2 (\varphi \otimes \varphi)\left(Z_{\theta(1)} \cdots Z_{\theta(k_0-1)} Z_{\theta(k_0+1)} \cdots Z_{\theta(m)}\right)
\end{align*}
for some $\theta : [m] \to [n]$ where $\theta(k) \neq \theta(k_0)$ for all $k \neq k_0$.  Since $(A_{\theta(k_0)}, 1)$ and $(1, B_{\theta(k_0)})$ are bi-free from each other and the other operators in $Z_{\theta(k)}$ for $k \neq k_0$, by using the bi-free moment-cumulant formula (\ref{eq:moment}), realizing the only non-zero bi-free cumulants occur when $A_{\theta(k_0)}$ and $B_{\theta(k_0)}$ are singletons thereby yielding their expecations, and reversing the bi-free moment-cumulant formulae without $A_{\theta(k_0)}$ and $B_{\theta(k_0)}$, we obtain
\begin{align*}
(\varphi \otimes \varphi) & \left( Z_{\theta(1)} \cdots Z_{\theta(k_0-1)} A_{\theta(k_0)} B_{\theta(k_0)} Z_{\theta(k_0+1)} \cdots Z_{\theta(m)}\right) \\
&= \varphi(a_{\theta(k_0)})\varphi(b_{\theta(k_0)}) (\varphi \otimes \varphi)\left(Z_{\theta(1)} \cdots Z_{\theta(k_0-1)} Z_{\theta(k_0+1)} \cdots Z_{\theta(m)}\right).
\end{align*}
Hence $(\varphi \otimes \varphi)(\pi) = 0$ when $\pi$ has a block with exactly one element.

If $\pi \in \P(m)$ does not contain a block with exactly one element and has a block with at least three elements, then $|\pi| < \frac{m}{2}$ and thus
\[
\lim_{n \to \infty} \frac{n(n-1) \cdots (n- |\pi| + 1)}{n^{\frac{m}{2}}} = 0.
\]
Combining the above, we obtain that
\[
\lim_{n \to \infty} (\varphi \otimes \varphi)(S^m_n) = \frac{1}{\delta^m} \sum_{\pi \in \P_2(m)} (\varphi \otimes \varphi)(\pi)
\]
where $\P_2(m)$ denotes the set of all pair partition on $m$.  

For all $m$ let 
\[
M'_m = \frac{1}{\delta^m} \sum_{\pi \in \P_2(m)} (\varphi \otimes \varphi)(\pi).
\]
To complete the proof, it suffices to prove that $M'_m = 0$ when $m$ is odd, $M'_2 = 1$, and $M'_m$ satisfies the recurrence relation in Corollary \ref{cor:recurrence-relation}.  Note clearly $M'_m = 0$ when $m$ is odd as $\P_2(m)$ is empty. 

We will now use bi-free independence to describe $M'_m$.  For notational simplicity, for $s\in \{1, e\}$ let
\[
L_k^s = \begin{cases}
A_k & \text{if $s = 1$} \\
-\lambda & \text{if $s = e$}
\end{cases} \qqand R_k^s = \begin{cases}
B_k & \text{if $s = 1$} \\
\lambda & \text{if $s = e$}
\end{cases}.
\]
For each $\pi\in \P_2(m)$ choose any $\theta_\pi : [m] \to [n]$ such that $\pi(\theta_\pi) = \pi$.  We view $\theta$ as a choice of colouring where two indices $k,j$ are coloured the same exactly when $\{k,j\} \in \pi$.  Using the above notation and bi-freeness, we obtain that
\begin{align*}
M'_m &= \frac{1}{\delta^m} \sum_{\pi \in \P_2(m)} \sum_{s : [m] \to \{1,e\}} (\varphi \otimes \varphi)\left(\prod_{k=1}^m L^{s(k)}_{\theta_\pi(k)}R^{s(k)}_{\theta_\pi(k)}\right) \\
&=  \frac{1}{\delta^m} \sum_{\pi \in \P_2(m)}\sum_{s : [m] \to \{1,e\}}     \sum_{\tau \in \BNC_{\vs}^a(m)} \kappa_\tau\left(L^{s(1)}_{\theta_\pi(1)},R^{s(1)}_{\theta_\pi(1)}, \ldots, L^{s(m)}_{\theta_\pi(m)},R^{s(m)}_{\theta_\pi(m)}\right).
\end{align*}
Instead of thinking of $\tau \in \BNC_{\vs}^a(m)$ as a partition on $1 < 2 < \cdots < 2m$, we will think of $\tau$ as a partition on $1_\ell < 1_r < 2_\ell < 2_r < \cdots < m_\ell < m_r$.  Thus, as $\pi \in \P_2(m)$, $\theta_\pi$ must colour $k_\ell$ and $k_r$ the same colour for all $k$ and each $k_\ell$ is the same colour of exactly one other $i_\ell$.  Thus, as mixed bi-free cumualants vanish, the only $\tau \in \BNC_{\vs}^a(m)$ that need be considered in the sum are those that can be coloured based on $\theta_\pi$.  For $\tau \in \BNC_{\vs}^a(m)$ we use $\tau \prec \pi$ to denote that $\tau$ can be coloured based on $\theta_\pi$.  Thus
\[
M'_m = \frac{1}{\delta^m} \sum_{\pi \in \P_2(m)}\sum_{s : [m] \to \{1,e\}}     \sum_{\substack{\tau \in \BNC_{\vs}^a(m) \\ \tau \prec \pi}} \kappa_\tau\left(L^{s(1)}_{\theta_\pi(1)},R^{s(1)}_{\theta_\pi(1)}, \ldots, L^{s(m)}_{\theta_\pi(m)},R^{s(m)}_{\theta_\pi(m)}\right).
\]
Therefore since $\{L_k^s, R_k^s \, \mid \, s \in \{1,e\}\}$ are identically distributed over $k$, we obtain that
\[
M'_m = \frac{1}{\delta^m} \sum_{\pi \in \P_2(m)} \sum_{\substack{\tau \in \BNC_{\vs}^a(m) \\ \tau \prec \pi}} \sum_{s : [m] \to \{1,e\}}  \kappa_\tau\left(L^{s(1)},R^{s(1)}, \ldots, L^{s(m)},R^{s(m)}\right)
\]
where $L^s = L_1^s$ and $R^s = R_1^s$ for all $s \in \{1,e\}$.

To see that $M'_2 = 1$, note
\begin{align*}
M'_2 &= \frac{1}{\delta^2} \sum_{\tau \in \BNC_{\vs}^a(2)} \sum_{s_1, s_2 \in \{1,e\}} \kappa_\tau\left(L^{s_1},R^{s_1}, L^{s_2},R^{s_2}\right)
\end{align*}
(i.e. the only pair partition on $[2]$ colours all nodes the same).  Note there are four elements of  $\BNC_{\vs}^a(2)$:
\begin{align*}
\tau_0 &= \{\{1_\ell\}, \{2_\ell\}, \{1_r\}, \{2_r\}\}, \\
\tau_\ell &= \{\{1_\ell, 2_\ell\}, \{1_r\}, \{2_r\}\},\\
\tau_r &= \{\{1_\ell\}, \{2_\ell\}, \{1_r, 2_r\}\}, \text{ and}\\
\tau_{\ell r } &= \{\{1_\ell, 2_\ell\}, \{1_r, 2_r\}\}.
\end{align*}
Since bi-free cumulants of order 1 are the just the moments and bi-free cumulants of order at least 2 involving scalars vanish by \cite{CNS2015-1}*{Proposition 6.4.1}, it is elementary to verify that
\begin{align*}
\sum_{s_1, s_2 \in \{1,e\}} \kappa_{\tau_0}\left(L^{s_1},R^{s_1}, L^{s_2},R^{s_2}\right) &= 0, \\
\sum_{s_1, s_2 \in \{1,e\}} \kappa_{\tau_\ell}\left(L^{s_1},R^{s_1}, L^{s_2},R^{s_2}\right)  &= \kappa_2(A_1) \kappa_1(B_1)^2 = \sigma^2\lambda^2, \\
\sum_{s_1, s_2 \in \{1,e\}} \kappa_{\tau_r}\left(L^{s_1},R^{s_1}, L^{s_2},R^{s_2}\right)  &= \kappa_1(A_1)^2 \kappa_2(B_1) = \lambda^2\sigma^2, \text{ and} \\
\sum_{s_1, s_2 \in \{1,e\}} \kappa_{\tau_{\ell r}}\left(L^{s_1},R^{s_1}, L^{s_2},R^{s_2}\right)  &= \kappa_2(A_1)\kappa_2(B_1) = \sigma^4.
\end{align*}
Thus
\[
M'_2 = \frac{1}{\delta^2} \left( \sigma^2\lambda^2 + \lambda^2\sigma^2  + \sigma^4 \right) = 1.
\]

To see that $M'_m$ satisfies the recurrence relation in Corollary \ref{cor:recurrence-relation}, we will apply a similar trick to that used in Corollary \ref{cor:recurrence-relation} by summing over $\pi \in \P_2(m)$ with the same block containing 1.  For a fixed $p \in [m]$ with $p \neq 1$, consider all $\pi \in \P_2(m)$ that contain the block $\{1,p\}$.  We will further divide this collection into two subcollections and sum bi-free cumulants over the subcollections separately.

First, consider all $\pi \in \P_2(m)$ that contain the block $\{1,p\}$ such that $\{k,k'\} \notin \pi$ for all $1 < k < p < k'$.  Thus there is a bijection between such $\pi$ and pairs of pair partitions $\pi_1$ and $\pi_2$ on $\{2, \ldots, p-1\}$ and $\{p+1, \ldots, m\}$ via $\pi_1 = \pi|_{\{2, \ldots, p-1\}}$ and $\pi_2 = \pi|_{\{p+1, \ldots, m\}}$.  Moreover, there is a bijection between $\tau \in \BNC^a_{\vs}(m)$ with $\tau \prec \pi$ and triples $(\tau', \tau_1, \tau_2)$ where $\tau' \in \{\tau_0, \tau_\ell, \tau_r, \tau_{\ell r}\}$, $\tau_1$ is any element of $\BNC^a_{\vs}(p-2)$ such that $\tau_1 \prec \pi_1$, and $\tau_2$ is any element of $\BNC^a_{\vs}(m-p)$ such that $\tau_2 \prec \pi_2$, via the map that sends $\tau$ to $\tau' = \tau|_{\{1_\ell, 1_r,p_\ell,p_r\}}$, $\tau_1 = \tau|_{\{2_\ell, \ldots, (p-1)_\ell, 2_r, \ldots, (p-1)_r\}}$, and $\tau_2 = \tau|_{\{(p+1)_\ell, \ldots, m_\ell, (p+1)_r, \ldots, m_r\}}$.  The following figure shows an example of this decomposition when $\pi = \{\{1,6\}, \{2,5\}, \{3,4\}, \{7,8\}, \{9,10\}\}$ so $p = 6$.

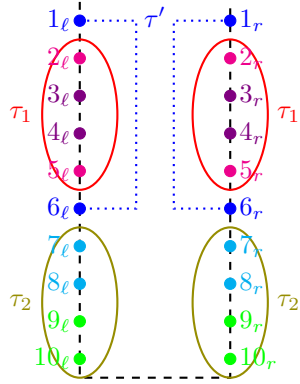
\begin{align*}
\begin{tikzpicture}[baseline]
	\draw[thick, dashed] (-1,4) -- (-1,-1) -- (1, -1) -- (1,4);
	\draw[thick, red] (1, 2.5) ellipse (.5 and 1);
	\node[right] at (1.5, 2.5) {\color{red}$\tau_1$};
	\draw[thick, red] (-1, 2.5) ellipse (.5 and 1);
	\node[left] at (-1.5, 2.5) {\color{red}$\tau_1$};
	\draw[thick, olive] (1, 0) ellipse (.5 and 1);
	\node[right] at (1.5, 0) {\color{olive}$\tau_2$};
	\draw[thick, olive] (-1, 0) ellipse (.5 and 1);
	\node[left] at (-1.5, 0) {\color{olive}$\tau_2$};
	\draw[fill=green, green] (-1,-0.75) circle (0.075);
	\node[left] at (-1, -0.75) {\color{green}$10_\ell$};
	\draw[fill=green, green] (-1,-0.25) circle (0.075);
	\node[left] at (-1, -0.25) {\color{green}$9_\ell$};
	\draw[fill=cyan, cyan] (-1, 0.25) circle (0.075);
	\node[left] at (-1, 0.25) {\color{cyan}$8_\ell$};
	\draw[fill=cyan, cyan] (-1, 0.75) circle (0.075);
	\node[left] at (-1, 0.75) {$\color{cyan}7_\ell$};	
	\draw[fill=blue, blue] (-1, 1.25) circle (0.075);
	\node[left] at (-1, 1.25) {\color{blue}$6_\ell$};
	\draw[fill=magenta, magenta] (-1, 1.75) circle (0.075);
	\node[left] at (-1, 1.75) {\color{magenta}$5_\ell$};	
	\draw[fill=violet, violet] (-1, 2.25) circle (0.075);
	\node[left] at (-1, 2.25) {\color{violet}$4_\ell$};
	\draw[fill=violet, violet] (-1, 2.75) circle (0.075);
	\node[left] at (-1, 2.75) {\color{violet}$3_\ell$};	
	\draw[fill=magenta, magenta] (-1, 3.25) circle (0.075);
	\node[left] at (-1, 3.25) {\color{magenta}$2_\ell$};
	\draw[fill=blue, blue] (-1, 3.75) circle (0.075);
	\node[left] at (-1, 3.75) {\color{blue}$1_\ell$};	
	\draw[fill=green, green] (1, -0.75) circle (0.075);
	\node[right] at (1, -0.75) {\color{green}$10_r$};
	\draw[fill=green, green] (1, -0.25) circle (0.075);
	\node[right] at (1, -0.25) {\color{green}$9_r$};
	\draw[fill=cyan, cyan] (1, 0.25) circle (0.075);
	\node[right] at (1, 0.25) {\color{cyan}$8_r$};
	\draw[fill=cyan, cyan] (1, 0.75) circle (0.075);
	\node[right] at (1, 0.75) {\color{cyan}$7_r$};	
	\draw[fill=blue, blue] (1, 1.25) circle (0.075);
	\node[right] at (1, 1.25) {\color{blue}$6_r$};
	\draw[fill=magenta, magenta] (1, 1.75) circle (0.075);
	\node[right] at (1, 1.75) {\color{magenta}$5_r$};	
	\draw[fill=violet, violet] (1, 2.25) circle (0.075);
	\node[right] at (1, 2.25) {\color{violet}$4_r$};
	\draw[fill=violet, violet] (1, 2.75) circle (0.075);
	\node[right] at (1, 2.75) {\color{violet}$3_r$};	
	\draw[fill=magenta, magenta] (1, 3.25) circle (0.075);
	\node[right] at (1, 3.25) {\color{magenta}$2_r$};
	\draw[fill=blue, blue] (1, 3.75) circle (0.075);
	\node[right] at (1, 3.75) {\color{blue}$1_r$};	
	\draw[dotted, thick, blue] (-1,3.75) -- (-.25, 3.75) -- (-.25, 1.25) -- (-1, 1.25);
	\draw[dotted, thick, blue] (1,3.75) -- (.25, 3.75) -- (.25, 1.25) -- (1, 1.25);
	\node at (0, 3.8) {\color{blue}$\tau'$};
\end{tikzpicture}
\end{align*}

Therefore, since
\[
\kappa_\tau\left(L^{s(1)},R^{s(1)}, \ldots, L^{s(m)},R^{s(m)}\right)
\]
will be the product of the corresponding bi-free cumulant terms for $\tau'$, $\tau_1$, and $\tau_2$, we obtain 
\[
\left( \sigma^2\lambda^2 + \lambda^2\sigma^2  + \sigma^4 \right) (\delta^{p-2}  M'_{p-2})  ( \delta^{m-p} M'_{m-p})  = \delta^m M'_{p-2} M'_{m-p}
\]
(where the terms come from $\tau'$, $\tau_1$, and $\tau_2$ respectively) by summing
\[
 \sum_{\substack{\tau \in \BNC_{\vs}^a(m) \\ \tau \prec \pi}} \sum_{s : [m] \to \{1,e\}}  \kappa_\tau\left(L^{s(1)},R^{s(1)}, \ldots, L^{s(m)},R^{s(m)}\right)
\]
over all $\pi \in \P_2(m)$ that contain the block $\{1,p\}$ such that $\{k,j\} \notin \pi$ for all $1 < k < p < j$.  Therefore, by summing over all possible $p$, the first term in the recurrence relation in Corollary \ref{cor:recurrence-relation} is obtained.

To complete the computation of $M'_m$, we need only consider $\pi \in \P_2(m)$ that contain the block $\{1,p\}$ such that there exists $1 < k < p < k'$ with $\{k,k'\} \in \pi$.  To aid in the comprehension of this argument, we will refer to the following figure as an example where $m = 16$, $p=13$, and
\[
\pi = \{\{1,13\}, \{2,3\}, \{4,12\}, \{5,6\}, \{7,16\}, \{8,11\}, \{9,10\}, \{14,15\}\}.
\]
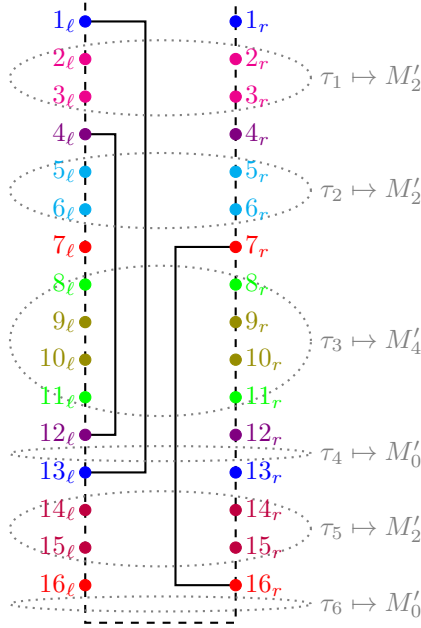
\begin{align*}
\begin{tikzpicture}[baseline]
	\draw[thick, dashed] (-1,4) -- (-1,-4.25) -- (1, -4.25) -- (1,4);
	\draw[thick] (-1, 3.75) -- (-0.2, 3.75) -- (-0.2, -2.25) -- (-1,-2.25);
	\draw[thick] (-1, 2.25) -- (-0.6, 2.25) -- (-0.6, -1.75) -- (-1, -1.75);
	\draw[thick] (1, -3.75) -- (0.2, -3.75) -- (0.2, 0.75) -- (1, 0.75);
	\draw[dotted, gray, thick] (0, 3) ellipse (2 and 0.5);
	\node[right] at (2, 3) {\color{gray} $\tau_1 \mapsto M'_2$};
	\draw[dotted, gray, thick] (0, 1.5) ellipse (2 and 0.5);
	\node[right] at (2, 1.5) {\color{gray} $\tau_2 \mapsto M'_2$};
	\draw[dotted, gray, thick] (0, -0.5) ellipse (2 and 1);
	\node[right] at (2, -0.5) {\color{gray} $\tau_3 \mapsto M'_4$};
	\draw[dotted, gray, thick] (0, -2) ellipse (2 and 0.1);
	\node[right] at (2, -2) {\color{gray} $\tau_4 \mapsto M'_0$};
	\draw[dotted, gray, thick] (0, -3) ellipse (2 and 0.5);
	\node[right] at (2, -3) {\color{gray} $\tau_5 \mapsto M'_2$};
	\draw[dotted, gray, thick] (0, -4) ellipse (2 and 0.1);
	\node[right] at (2, -4) {\color{gray} $\tau_6 \mapsto M'_0$};
	\draw[fill=red, red] (-1,-3.75) circle (0.075);
	\node[left] at (-1, -3.75) {\color{red}$16_\ell$};
	\draw[fill=purple, purple] (-1,-3.25) circle (0.075);
	\node[left] at (-1, -3.25) {\color{purple}$15_\ell$};
	\draw[fill=purple, purple] (-1,-2.75) circle (0.075);
	\node[left] at (-1, -2.75) {\color{purple}$14_\ell$};
	\draw[fill=blue, blue] (-1,-2.25) circle (0.075);
	\node[left] at (-1, -2.25) {\color{blue}$13_\ell$};
	\draw[fill=violet, violet] (-1,-1.75) circle (0.075);
	\node[left] at (-1, -1.75) {\color{violet}$12_\ell$};
	\draw[fill=green, green] (-1,-1.25) circle (0.075);
	\node[left] at (-1, -1.25) {\color{green}$11_\ell$};
	\draw[fill=olive, olive] (-1,-0.75) circle (0.075);
	\node[left] at (-1, -0.75) {\color{olive}$10_\ell$};
	\draw[fill=olive, olive] (-1,-0.25) circle (0.075);
	\node[left] at (-1, -0.25) {\color{olive}$9_\ell$};
	\draw[fill=green, green] (-1, 0.25) circle (0.075);
	\node[left] at (-1, 0.25) {\color{green}$8_\ell$};
	\draw[fill=red, red] (-1, 0.75) circle (0.075);
	\node[left] at (-1, 0.75) {$\color{red}7_\ell$};	
	\draw[fill=cyan,cyan] (-1, 1.25) circle (0.075);
	\node[left] at (-1, 1.25) {\color{cyan}$6_\ell$};
	\draw[fill=cyan, cyan] (-1, 1.75) circle (0.075);
	\node[left] at (-1, 1.75) {\color{cyan}$5_\ell$};	
	\draw[fill=violet, violet] (-1, 2.25) circle (0.075);
	\node[left] at (-1, 2.25) {\color{violet}$4_\ell$};
	\draw[fill=magenta, magenta] (-1, 2.75) circle (0.075);
	\node[left] at (-1, 2.75) {\color{magenta}$3_\ell$};	
	\draw[fill=magenta, magenta] (-1, 3.25) circle (0.075);
	\node[left] at (-1, 3.25) {\color{magenta}$2_\ell$};
	\draw[fill=blue, blue] (-1, 3.75) circle (0.075);
	\node[left] at (-1, 3.75) {\color{blue}$1_\ell$};	
	\draw[fill=red, red] (1,-3.75) circle (0.075);
	\node[right] at (1, -3.75) {\color{red}$16_r$};
	\draw[fill=purple, purple] (1,-3.25) circle (0.075);
	\node[right] at (1, -3.25) {\color{purple}$15_r$};
	\draw[fill=purple, purple] (1,-2.75) circle (0.075);
	\node[right] at (1, -2.75) {\color{purple}$14_r$};
	\draw[fill=blue, blue] (1,-2.25) circle (0.075);
	\node[right] at (1, -2.25) {\color{blue}$13_r$};
	\draw[fill=violet, violet] (1,-1.75) circle (0.075);
	\node[right] at (1, -1.75) {\color{violet}$12_r$};
	\draw[fill=green, green] (1,-1.25) circle (0.075);
	\node[right] at (1, -1.25) {\color{green}$11_r$};
	\draw[fill=olive, olive] (1, -0.75) circle (0.075);
	\node[right] at (1, -0.75) {\color{olive}$10_r$};
	\draw[fill=olive, olive] (1, -0.25) circle (0.075);
	\node[right] at (1, -0.25) {\color{olive}$9_r$};
	\draw[fill=green, green] (1, 0.25) circle (0.075);
	\node[right] at (1, 0.25) {\color{green}$8_r$};
	\draw[fill=red, red] (1, 0.75) circle (0.075);
	\node[right] at (1, 0.75) {\color{red}$7_r$};	
	\draw[fill=cyan,cyan] (1, 1.25) circle (0.075);
	\node[right] at (1, 1.25) {\color{cyan}$6_r$};
	\draw[fill=cyan, cyan] (1, 1.75) circle (0.075);
	\node[right] at (1, 1.75) {\color{cyan}$5_r$};	
	\draw[fill=violet, violet] (1, 2.25) circle (0.075);
	\node[right] at (1, 2.25) {\color{violet}$4_r$};
	\draw[fill=magenta, magenta] (1, 2.75) circle (0.075);
	\node[right] at (1, 2.75) {\color{magenta}$3_r$};	
	\draw[fill=magenta, magenta] (1, 3.25) circle (0.075);
	\node[right] at (1, 3.25) {\color{magenta}$2_r$};
	\draw[fill=blue, blue] (1, 3.75) circle (0.075);
	\node[right] at (1, 3.75) {\color{blue}$1_r$};	
\end{tikzpicture}
\end{align*}

Let $C$ be the connected component of the intersection graph of $\pi$ containing $\{1,p\}$ and let $I$ denote the set of all indices contained in the blocks in $C$.  Then $|I| = 2j$ for some $j \geq 2$ and $\pi|_{I} \in \P_2^{\mathrm{con}}(2j)$.  For the above figure, $C$ is the graph with vertices $\{1,13\}$, $\{4,12\}$, $\{7,16\}$ where the only edges are from $\{7,16\}$ to $\{1,13\}$ and from $\{7,16\}$ to $\{4,12\}$, $\pi_I = \{\{1,13\}, \{4,12\}, \{7,16\} \}$, and $j = 3$.

Suppose $\{k,k'\}, \{h,h'\} \in \pi$ are such that $k < h < k' < h'$.  For $\tau \in \BNC^a_{\vs}(m)$ such that $\tau \prec \pi$, note that
\[
\sum_{s_1, s_2 \in \{1,e\}} \kappa_\tau\left(L^{s(1)},R^{s(1)}, \ldots, L^{s(m)},R^{s(m)}\right) = 0
\]
when $\tau|_{\{k_\ell, k_r,k'_\ell,k'_r\}} = \tau_0$ and when $\tau|_{\{k_\ell, k_r,k'_\ell,k'_r\}} = \tau_{\ell r}$ as this latter condition implies $\tau|_{\{h_\ell, h_r,h'_\ell,h'_r\}} = \tau_0$ since $\tau$ is bi-non-crossing.  Moreover, since $\tau$ is bi-non-crossing, it is not possible that $\tau|_{\{k_\ell, k_r,k'_\ell,k'_r\}}$ and $\tau|_{\{h_\ell, h_r,h'_\ell,h'_r\}}$ are both $\tau_\ell$ or both $\tau_r$.  Hence the sum can only be non-zero if $ \{\tau|_{\{k_\ell, k_r,k'_\ell,k'_r\}}, \tau|_{\{h_\ell, h_r,h'_\ell,h'_r\}}  \}=  \{\tau_\ell, \tau_r\}$.  Therefore, if $\pi|_{I}$ is not bipartite, then the above sum is 0 for all $\tau \in \BNC^a_{\vs}(m)$ such that $\tau \prec \pi$.  Hence we need only consider $\pi$ such that $\pi|_{I} \in \P_2^{\mathrm{bicon}}(2j)$ for the remainder of our computations.  Moreover, when $\pi|_{I} \in \P_2^{\mathrm{bicon}}(2j)$ there are exactly two possible options for $\tau|_{\{i_\ell, i_r\}_{i \in I}}$ for the above sum to be non-zero (i.e. $\tau|_{\{1_\ell, 1_r, p_\ell, p_r\}}$ is either $\tau_\ell$ or $\tau_r$ and we alternate over the bipartite graph $C$).  In the above figure, the solid lines show one of the two options where $\tau_\ell$ occurs for $\{1,13\}$.  Note in the above figure that it is not possible for an index in the second oval from the top to be the same colour as an index in the second oval from the bottom as this would result in $C$ being not bi-partite.

Write $I = \{1 = i_1 < i_2 < \cdots < i_{2j}\}$ and let $k_h = i_{h+1} - i_{h}-1$ for $1 \leq h \leq 2j$ where $i_{2j+1} = m+1$.  Then $0 \leq k_1, \ldots, k_{2j} \leq m-2j$ and $k_1 + \cdots + k_{2j} = m-2j$.  For the above figure, $i_2 = 4$, $i_3 = 7$, $i_4 = 12$, $i_5 = 13$, and $i_6 = 16$ so $k_1 = 2$, $k_2 = 2$, $k_3 = 4$, $k_4 = 0$, $k_5 = 2$, and $k_6 = 0$.

We now desire to sum
\[
\sum_{\substack{\tau \in \BNC_{\vs}^a(m) \\ \tau \prec \pi}} \sum_{s : [m] \to \{1,e\}}  \kappa_\tau\left(L^{s(1)},R^{s(1)}, \ldots, L^{s(m)},R^{s(m)}\right)
\]
over all $\pi$ with the same $C$ (and thus same $I$).  Note there are exactly 2 possible options for $\tau|_{\{i_\ell, i_r\}_{i \in I}}$.  Moreover, for since $\pi|_{\{i_h+1, \ldots, i_{h+1}-1\}}$ is a pair partition for each $k$, we see that $\tau|_{\{(i_h+1)_\ell, \ldots, (i_{h+1}-1)_\ell, (i_h+1)_r, \ldots, (i_{h+1}-1)_r \}}$ can be any element of $\BNC_{\vs}^a(k_h)$.  In the above figure, these $2j$ partitions are represented by the oval regions.

By summing over all $\BNC_{\vs}^a(k_h)$ and then over the options of $\tau|_{\{i_\ell, i_r\}_{i \in I}}$, we obtain that the above sum is
\[
2  \left(\sigma^2 \lambda^2\right)^j (\delta^{k_1}M'_{k_1}) \cdots  (\delta^{k_{2j}}M'_{k_{2j}}) = 2 \delta^{m} \left(\frac{\sigma^2 \lambda^2}{\delta^2}\right)^{j} M'_{k_1} \cdots M'_{k_{2j}}
\]
where the 2 comes from the two options of $\tau|_{\{i_\ell, i_r\}_{i \in I}}$, a $\sigma^2 \lambda^2$ occurs for each $\tau_\ell$ and $\tau_r$ that occurs in $\tau|_{\{i_\ell, i_r\}_{i \in I}}$ (so there are $\frac{|I|}{2} = j$ occurrences), and each $M'_{k_h}$ occurs from the sum over all options of $\tau|_{\{(i_h+1)_\ell, \ldots, (i_{h+1}-1)_\ell, (i_h+1)_r, \ldots, (i_{h+1}-1)_r \}}$.  Hence, by summing
\[
\sum_{\substack{\tau \in \BNC_{\vs}^a(m) \\ \tau \prec \pi}} \sum_{s : [m] \to \{1,e\}}  \kappa_\tau\left(L^{s(1)},R^{s(1)}, \ldots, L^{s(m)},R^{s(m)}\right)
\]
over all $\pi$ with the same $I$, we obtain 
\[
2 \delta^{m} \left(\frac{\sigma^2 \lambda^2}{\delta^2}\right)^{j} |\P_2^{\mathrm{bicon}}(2j) | M_{k_1} \cdots M_{k_{2j}}
\]
as there are $|\P_2^{\mathrm{bicon}}(2j)|$ options for $\pi|_I$.    Therefore, since
\[
\frac{\sigma^2\lambda^2}{\delta^2} = \frac{q}{2},
\]
by summing over all options of $I$, we obtain that
\begin{align*}
M'_m &=  \frac{1}{\delta^m} \sum_{\pi \in \P_2(m)} \sum_{\substack{\tau \in \BNC_{\vs}^a(m) \\ \tau \prec \pi}} \sum_{s : [m] \to \{1,e\}}  \kappa_\tau\left(L^{s(1)},R^{s(1)}, \ldots, L^{s(m)},R^{s(m)}\right)\\
&= \sum_{\substack{0 \leq k_1, k_2 \leq m-2 \\ k_1 + k_2 = m-2}} M'_{k_1} M'_{k_2} + \sum_{j \geq 2} \sum_{\substack{0 \leq k_1, \ldots, k_{2j} \leq m-2j \\ k_1 + \cdots + k_{2j} = m-2j}} 2\left(\frac{q}{2}\right)^{j} |\P_2^{\mathrm{bicon}}(2j) | M_{k_1} \cdots M_{k_{2j}}.
\end{align*}
Hence the result follow from Corollary \ref{cor:recurrence-relation}.
\end{proof}

\end{document}